\documentclass[11pt, a4paper,reqno]{amsart}
\usepackage[a4paper, margin=4cm]{}
\usepackage{amssymb,amsfonts, amsthm, xfrac, amscd}
\usepackage{setspace}
\usepackage{breqn}
\usepackage{changes}

\usepackage{color}
\usepackage{array}

\usepackage{color}
\usepackage[bookmarks]{hyperref}

\allowdisplaybreaks[4]

\newtheoremstyle{myremark}     {10pt}{10pt}{}{}{\bfseries}{.}{.5em}{}

\newtheorem{thm}{Theorem}[section]
\newtheorem{cor}[thm]{Corollary}
\newtheorem{lem}[thm]{Lemma}
\newtheorem{pro}[thm]{Proposition}
\theoremstyle{definition}
\newtheorem{defn}[thm]{Definition}
\newtheorem{exmp}[thm]{Example}
\theoremstyle{myremark}
\newtheorem{rem}[thm]{Remark}

\numberwithin{equation}{section}
\begin{document}


	\title[Covariance identities and Variance bounds]{Covariance identities and Variance bounds for Infinitely Divisible Random Variables and their Applications}

	\author[Barman]{Kalyan Barman}
	\address{\hskip-\parindent
		Kalyan Barman, Department of Mathematics, IIT Bombay,
		Powai - 400076, India.}
	
	\email{barmankalyan@math.iitb.ac.in}
	
	\author[Upadhye]{Neelesh S Upadhye}
	
	\address{\hskip-\parindent
		Neelesh S Upadhye, Department of Mathematics, IIT Madras,
		Chennai - 600036, India.}
	\email{neelesh@iitm.ac.in}
	
	\author[Vellaisamy]{Palaniappan Vellaisamy}
	\address{\hskip-\parindent
		P Vellaisamy, Department of Statistics and Applied Probability, UC Santa Barbara, Santa Barbara, CA, 93106, USA}
	\email{pvellais@ucsb.edu}
	
	\date{\textit{August 2, 2024}}
	
	\begin{abstract}
		In this article, we establish a general covariance identity for infinitely divisible distributions ($IDD$). Using this result, we derive Cacoullos type variance bounds for the $IDD$. Applications to some important distributions are discussed, in addition to the computation of variance bounds for certain posterior distributions. As another application, we derive the Stein-type identity for the $IDD$, which involves the L\'evy measure. This result in turn is used to derive the Stein-type identity for the $CGMY$ distributions and the variance-gamma distributions ($VGD$). This approach, especially for the $VGD$ is new and simpler, compared to the ones available in the literature. Finally, as another nontrivial application, we apply the covariance identity in deriving known and some new formulas for the weighted premium calculation principles ($WPCP$) and Gini coefficient for the $IDD$.
	\end{abstract}

	\maketitle
	
\section{Introduction} \label{Intro}	

\noindent
In 1972, Charles Stein, while attempting a new proof of central limit theorem, invented a new identity for normal distribution and used it to study the normal approximation problems for the sums of dependent random variables (rvs) (see, Stein \cite{k2}). The invention of similar identities for other probability distributions is well studied in the literature, with applications to limit theorems \cite{barman96,ce0,novak,k1}, runs \cite{k7,vellai1}, estimation theory \cite{BR,Zhu}, functional inequalities \cite{barman1}, insurance \cite{geoP} and various other fields. 

\vskip 2ex
\noindent
 In particular, Stein \cite{k2} proved that for a normal random variable (rv) $X\sim \mathcal{N}(\mu, \sigma^{2})$, 
\begin{align}
	\text{Cov}(X,g(X))=\sigma^{2}\mathbb{E}[g^{\prime}(X)], \label{covidentity1}
\end{align}
\noindent
where $g$ is an absolutely continuous function with $\mathbb{E}[g^{\prime}(X)] < \infty$. 

\vskip 2ex
\noindent
Goldstein and Reinert \cite{goldren0} generalized the Stein identity in \eqref{covidentity1} as follows. Let $X$ be a real-valued rv with mean zero and variance $\sigma^{2}$. Then, for any differentiable function $g$ with $\mathbb{E}(Xg(X))<\infty$, there exists a rv $X^{*}$ having the unimodal probability density function ($pdf$)
\begin{align*}
	f_{X^*}(x)=\frac{1}{\sigma^{2}}\mathbb{E} \left(X \mathbf{I}(X>x) \right)
\end{align*}
\noindent
such that
\begin{align}\label{reinertcovid}
	Cov (X,g(X))=\sigma^{2}\mathbb{E}\left( g^{\prime} (X^{*})\right).
\end{align}
\noindent
We note that, from \cite[Lemma 2.1(ii)]{goldren0}, $X^*$ is supported on the closed convex hull of the support of $X$. Also, they applied the covariance identity \eqref{reinertcovid} to obtain the rate of convergence in the central limit theorem. Further, they presented a nice application to dependent samples.

\noindent
Later, Papadatos and Papathanasiou \cite{papa} established a covariance identity, similar to \eqref{reinertcovid}, for an absolutely continuous rv $X$ given by
\begin{align}\label{covidentity2}
\text{Cov}(X,g(X))=\sigma^{2} \mathbb{E}[g^{\prime}(X^{*})],
\end{align}
\noindent
where the rv $X^*$ has the unimodal density
\begin{align}\label{papapdf}
f_{X^{*}}(x)=\frac{1}{\sigma^{2}} \int_{x}^{\infty}[t-\mathbb{E}(X)]f_{X}(t)dt.	
\end{align}

\noindent
From \eqref{covidentity2}, it is clear that, if $X\sim \mathcal{N}(\mu, \sigma^{2})$, then $X^{*}\sim \mathcal{N}(\mu, \sigma^{2})$. Observe that the Stein-type covariance identity \eqref{covidentity2} depends on the $pdf$ of $X$. The derivation of the identity becomes difficult, whenever $f_X(x)$ is not in the closed form. 

\vskip 2ex
\noindent
Let $X$ follow the infinitely divisible distribution ($IDD$), denoted by $IDD(\mu,\sigma^2,\nu)$, with parameters $\mu,\sigma^2$ and the associated L\'evy measure $\nu$ (see \eqref{nkey1}). It is known that the $pdf$ of some distributions belonging to $IDD$ family may not be available in a closed form. Therefore, the derivation of Stein-type covariance identity is not possible for all the distributions belonging to $IDD$ family using the approach given in \cite{papa}. In 1998, Houdr\'e {\it et al.} \cite{covrep} established a covariance representation for $IDD$ via the generator approach. They applied it to study the association problems for $IDD$ and also correlation inequalities. Recently, Arras and Houdr\'e \cite{arras0} obtained a Stein-type covariance identity, like \eqref{covidentity1} for $IDD$ with first finite moment using covariance representation given in \cite[Proposition 2]{covrep}. More recently, Upadhye and Barman \cite[Theorem 3.1]{k1} established a Stein-type identity for $IDD$ via the characteristic function (cf) approach. 

\vskip 2ex
\noindent
	In this article, we first obtain an extended Stein-type covariance identity in Theorem \ref{qtheorem1} for the $IDD$. The identity for the positive $IDD$ is illustrated for the first time, in terms of cumulants, to the best of our knowledge. Although covariance identities for $IDD$ are well studied (see, \cite{houdrecovrep} and \cite{covrep}), the covariance identity in Theorem \ref{qtheorem1} is useful. Using this result, we establish Cacoullos type variance bounds which are used to compute variance bounds of the parameters of the posterior distributions. Also, we obtain Stein identity for the variance-gamma distributions ($VGD$) using covariance identity and the associated L\'evy measure. Our approach is new and much simpler. We finally apply the covariance identity in deriving some new formulas of weighted premium calculation principles ($WPCP$), that is, for $\mathbb{E}[Xw(X)]\big/\mathbb{E}[w(X)]$, in terms of L\'evy measure. Also, the Gini coefficient for $IDD$ is derived. Observe that the density of the $IDD$ is not usually available in explicit form, and in such case the approach based on the L\'evy measure is quite useful. Several well chosen examples illustrate our methodology, giving simpler derivation of the known results as well as some new formulas.
	

\vskip 2ex	
\noindent
The article is organized as follows. In Section \ref{pre}, we give the L\'evy-Khintchine representation of $IDD(\mu,\sigma^{2},\nu)$. Note that the class of $IDD(\mu,\sigma^{2},\nu)$ is quite large and it includes many subclasses of distributions, such as compound Poisson distributions ($CPD$) and generalized tempered stable distributions ($GTSD$). Also, we discuss rather in detail some special cases of the $GTSD$, which include variance-gamma distributions ($VGD$), bilateral-gamma distributions ($BGD$) and Carr, Geman, Madan and Yor ($CGMY$) distributions. In Section \ref{mainr}, we derive an interesting covariance identity for $IDD(\mu,0,\nu)$ (see Theorem \ref{qtheorem1}), which depends on the L\'evy measure $\nu$. Using this result, we obtain new covariance identities for $GTSD$ and $VGD$. In Section \ref{vbbd}, using our covariance identity, given in Section \ref{mainr}, we derive Cacoullos-type variance bounds for $g(X)$, where $X\sim IDD(\mu,0,\nu)$. Applications to gamma, Laplace and two-sided exponential distributions are considered. Also, as another application to the Bayesian inference, we find variance bounds for the posterior distributions. In Section \ref{appofre}, we consider two important applications of our results. The first one is concerned with obtaining Stein identities via the covariance identity. We consider the $VGD$ and derive the Stein identity using the covariance identity and the L\'evy measure of the $VGD$. It can be seen that our approach involves much simpler calculations, compared to the one used in \cite{k24} where the density approach is used, and it involves lengthy calculations and modified Bessel functions. Our second application is the calculation of $WPCP$. First, we obtain a result for the $WPCP$, when $X\sim IDD(\mu,0,\nu)$. Using this result, we compute $WPCP$ for gamma, $CPD$, $BGD$, $VGD$, inverse Gaussian and $CGMY$ distributions. Finally, we compute also the Gini coefficient for $IDD(\mu,0,\nu)$ and its applications to $VGD$ and $CGMY$ distributions are discussed.

\section{Notations and Preliminary Results}\label{pre}
	\noindent
	 Let $\textbf{I}_{A}(.)$ denote the indicator function of the set $A$. Let $X$ be an infinitely divisible rv. Then its L\'evy-Khintchine representation of cf (see \cite{sato}) is given by 
	 \begin{equation}\label{nkey1}
			\phi_{X}(t)=\exp \left(it\mu_0  -\frac{\sigma^{2}t^{2}}{2}+\int_{\mathbb{R}}(e^{itu}-1-itu\mathbf{I}_{\{|u|\leq 1\}}(u))\nu(du)\right),~~t\in\mathbb{R}, 
		\end{equation}
		where $\mu_{0} \in \mathbb{R}$, $\sigma\geq 0$ and $\nu$, the L\'evy measure on $\mathbb{R}\setminus \{0\}$ satisfying $\int_{\mathbb{R}}(1 \wedge u^{2})\nu(du)<\infty.$ 	
	\noindent
	Observe that if $\int_{\{|u|>1\}}u\nu(du) < \infty,$ then \eqref{nkey1} can be written as 	
	\begin{align}\label{qchfgtsd}
		\phi_{X}(t)&=\exp \left(it\mu -\frac{\sigma^{2}t^{2}}{2}+\int_{\mathbb{R}}\left(e^{itu} -1-itu\right)\nu(du)\right) ,~~t\in\mathbb{R},
	\end{align}
	\noindent
	where $\mu=\mu_{0}+\int_{\{|u|>1\}}u\nu(du)$. For a rv $X$ with cf \eqref{qchfgtsd}, we write $X\sim IDD(\mu, \sigma^{2}, \nu)$. 
	
	\vskip 2ex
	\noindent
	If $X \sim IDD(\mu,\sigma^{2},\nu)$, then its $n$-th cumulant is
	
	\vfill
	
	\begin{align}
		C_{n}(X):=(-i)^{n}\frac{d^{n}}{dt^{n}}\log\phi_{X}(t)\bigg|_{t=0},~n\geq 1. 
	\end{align}

	\noindent
	 Let $X\sim IDD(\mu,\sigma^{2},\nu)$ have moments of arbitrary order. Then
	\begin{align}
C_{1}(X)&=\mathbb{E}(X)=\mu ,\\
\label{cumul}
C_{2}(X)&=Var(X)=
\sigma^{2}+
\int_{\mathbb{R}}u^{2}\nu(du), \text{ and }\\
C_{k}(X)&=\int_{\mathbb{R}}u^{k}\nu(du),~k\geq 3.	
\end{align}
	\noindent
	The class of $IDD$ is quite large, and see \cite{sato} for more properties. Next, we discuss some important subclasses of $IDD$, which are useful in later sections.
	\subsection{Compound Poisson distributions} 
	A rv $X$ is said to have $CPD$ if its cf (see \cite{sato}) is given by
	\begin{equation}\label{cfcp}
	\phi_{cp}(t)=\exp\left(\nu(\mathbb{R})\int_{\mathbb{R}}(e^{itu}-1)\nu_0(du)  \right),~~t \in \mathbb{R},
	\end{equation}
	where the L\'evy measure $\nu$ is finite i.e., $\nu(\mathbb{R})<\infty$ and $\nu_0$ is the Borel probability measure on $\mathbb{R}$, defined by $\nu_0(du)=\frac{\nu(du)}{\nu(\mathbb{R})}$. We denote it by $CPD(\nu(\mathbb{R}),\nu_0)$. Let $\overset{d}{=}$ denotes equality in distribution. Note that $CPD(\nu(\mathbb{R}),\nu_0)\overset{d}{=}IDD(\mu,0,\nu),$ where $\mu=\int_{\mathbb{R}}u\nu(du)$. Let $\delta_{1}$ be the Dirac measure concentrated at $1$. When $\lambda =\nu (\mathbb{R})$ and $\nu_0(du)=\delta_1(du)$, then $CPD(\lambda,\lambda \delta_{1})$ has Poisson distribution with mean $\lambda$, denoted by $Poi(\lambda)$.

\vskip 1ex
	\noindent
	Next, we introduce $GTSD$ and discuss some of their relevant properties (see \cite{cont}).
	\subsection{Generalized tempered stable distributions}
	A rv $X$ is said to have $GTSD$ if its cf (see \cite[Section 4.5]{cont}) is given by
	\begin{align}\label{nep5}
		\phi_{\text{gts}}(t)
		&=\exp \left(it\mu +\int_{\mathbb{R}}\left(e^{itu} -1-itu\right)\nu_{gts}(du)\right) ,~~t\in\mathbb{R},
		\end{align}		
	\noindent	
	where the L\'evy measure $\nu_{gts}$ is

	\begin{equation}\label{nep6}
		\nu_{gts}(du)=\left(\frac{\alpha^{+}  }{u^{1+\beta}}e^{-\lambda^{+}u}\mathbf{I}_{(0,\infty)}(u)+\frac{\alpha^{-}  }{|u|^{1+\beta}}e^{-\lambda^{-}|u|}\mathbf{I}_{(-\infty,0)}(u)\right)du,
	\end{equation}
\noindent
and the parameters $\mu \in \mathbb{R}$, $\alpha^{+}, \lambda^{+},\alpha^{-},\lambda^{-} \in (0,\infty),$ and $\beta \in [0,2).$ We denote it by

\noindent
 $GTSD(\mu,\beta,\alpha^{+},$ $\lambda^{+},\alpha^{-},\lambda^{-})$. Note that $GTSD(\mu,\beta,\alpha^{+},\lambda^{+},\alpha^{-},\lambda^{-}) \overset{d}{=}IDD(\mu,0,\nu_{gts})$. 

\vskip 1ex
\noindent
Next, we list some special distributions of the $GTSD$ family.  
\subsubsection{Bilateral-gamma distributions} A rv $X$ is said to have $BGD$ if its cf (see \cite{kuchtapbiden}) is given by
	\begin{equation}\label{bgdcf}
	\phi_{bg}(t)=\exp\left\{\int_{\mathbb{R}}(e^{itu}-1)\nu_{bg}(du) \right\},~~t \in \mathbb{R},
	\end{equation}
	\noindent
	where the L\'evy measure $\nu_{bg}$ is

	\begin{equation}\label{bgdlevy}
	\nu_{bg}(du)=\left(\frac{\alpha^{+}  }{u}e^{-\lambda^{+}u}\mathbf{I}_{(0,\infty)}(u)+\frac{\alpha^{-}  }{|u|}e^{-\lambda^{-}|u|}\mathbf{I}_{(-\infty,0)}(u)\right)du,
	\end{equation}
	\noindent
	and the parameters $\alpha^{+},\lambda^{+},\alpha^{-}, \lambda^{+} \in (0,\infty).$ We denote it by $BGD(\alpha^{+},\lambda^{+},\alpha^{-},$ $\lambda^{-})$. Note that $BGD(1,\lambda^{+},1,$ $\lambda^{-})$ has two-sided exponential distribution (see \cite{sato}). Note also that $BGD(\alpha^{+},\lambda^{+},$ $\alpha^{-},\lambda^{-}) \overset{d}{=}GTSD(\mu,0,\alpha^{+},\lambda^{+},\alpha^{-},\lambda^{-})$, where $\mu=\int_{\mathbb{R}}u\nu_{bg}(du).$

	\noindent
	Let now $X_1$ and $X_2$ be two independent gamma $Ga(\alpha^+,\lambda^+)$ and $Ga(\alpha^-,\lambda^-)$ rvs, where $Ga(\alpha,\lambda)$ has the density
	\begin{align*}
		f(x)=\frac{\lambda^{\alpha}}{\Gamma(\alpha)}x^{\alpha-1}e^{-\lambda x},~x>0. 
	\end{align*}
	\noindent
	 Then  $X_1-X_2 \sim BGD(\alpha^{+},\lambda^{+},\alpha^{-},\lambda^{-}).$ It is known that the $pdf$ of $BGD$ is symmetric about the origin and its symmetric density \cite[Section 3]{kuchtapbiden} is given by
	\begin{align}
		f_{X}(x)=\frac{(\lambda^{+})^{\alpha^{+}} (\lambda^{-})^{\alpha^{-}}}{(\lambda^{+}+\lambda^{-})^{\alpha^{-}}\Gamma (\alpha^{+})\Gamma (\alpha^{-}) } \displaystyle\int_{0}^{\infty}v^{\alpha^{-} -1} \left( x+ \frac{v}{\lambda^{+}+\lambda^{-}} \right)^{\alpha^{+} -1} e^{-v} dv,
	\end{align} 
	\noindent
	$x\in \mathbb{R}\setminus \{ 0\}$. We refer the reader to \cite{bilateral} for more properties of $BGD$.
	\subsubsection{Variance-gamma distributions} A rv $X$ is said to have $VGD$ if its cf (see \cite{kk1}) is given by
	\begin{align}
	\label{eqvg}\phi_{vg}(t)&=e^{it\mu_0}\left(1-it\left(\frac{1}{\lambda^{+}}-\frac{1}{\lambda^{-}}\right)+\frac{t^{2}}{\lambda^{+}\lambda^{-}}  \right)^{-\alpha} \\
	\label{cfvgd1}&=\exp\left\{it\mu_0+ \int_{\mathbb{R}}(e^{itu}-1)\nu_{vg}(du) \right\},~~t\in \mathbb{R},
	\end{align}
	\noindent
	where the L\'evy measure $\nu_{vg}$ is
	\begin{equation}\label{levyvgd}
	\nu_{vg}(du)=\left(\frac{\alpha  }{u}e^{-\lambda^{+}u}\mathbf{I}_{(0,\infty)}(u)+\frac{\alpha  }{|u|}e^{-\lambda^{-}|u|}\mathbf{I}_{(-\infty,0)}(u)\right)du,
	\end{equation}
	\noindent
	and the parameters $\mu_0\in \mathbb{R}$, $\alpha,\lambda^{+},\lambda^{-} \in (0,\infty).$ We denote it by $VGD_0(\mu_{0},\alpha,\lambda^{+},$ $\lambda^{-})$.

	\noindent
	Note that $VGD_0(\mu_{0},\alpha,\lambda^{+},\lambda^{-}) \overset{d}{=}GTSD(\mu,0,\alpha,\lambda^{+},\alpha,\lambda^{-})$, where $\mu-\int_{\mathbb{R}}u\nu_{vg}(du)$ $=\mu_0.$ Note also that $VGD_0(\mu_{0},1,\frac{1}{\delta},\frac{1}{\delta})$ has Laplace $La(\mu_{0},\delta^2)$ distribution. If we set $\frac{1}{\lambda^{+}\lambda^{-}}=\sigma^{2}, \left(\frac{1}{\lambda^{+}}-\frac{1}{\lambda^{-}}\right)=2\theta$ and $\alpha=\frac{r}{2}$ in (\ref{eqvg}), we get
	\begin{equation}\label{cfvg2}
	\phi_{vg}(t)=e^{it\mu_0}\left(1-i2\theta t+\sigma^{2}t^{2}\right)^{-\frac{r}{2}},
	\end{equation}
	where $\mu_0, \theta\in\mathbb{R}$ and $\sigma^{2},r \in (0,\infty)$, and it is denoted by $VGD_{1}(\mu_0,\sigma^{2},r,\theta)$ (see \cite[eqn. 1.10]{gammadiff}). It is known \cite[Section 1]{kk1} that the $pdf$ of $VGD_{1}(\mu_0,\sigma^{2},r,\theta)$ is given by 
		
		\begin{align}
			f_{X}(x)=\frac{1}{\sigma \sqrt{\pi} \Gamma (\frac{r}{2})}e^{\frac{\theta}{\sigma^{2}} (x-\mu_{0}) } \left( \frac{ |x-\mu_{0}|}{2 \sqrt{\theta^{2} +\sigma^{2} }}  \right)^{\frac{r-1}{2}}K_{\frac{r-1}{2}} \left( \frac{\sqrt{\theta^{2} +\sigma^{2} }}{\sigma^{2}} |x-\mu_{0}| \right),
		\end{align}
		\noindent
		where $x\in \mathbb{R},$ and $K_{\zeta}(x)$ is the modified Bessel function of the second kind, given by
		
		$$K_{\zeta}(x)=\frac{1}{2} \int_{0}^{\infty} z^{\zeta -1}e^{- \frac{x}{2}(z +\frac{1}{z})}dz.$$

			\subsubsection{CGMY distributions} A rv $X$ is said to have $CGMY$ distribution if its cf (see \cite{k13}) is given by
			\begin{equation}\label{CGMYcf}
			\phi_{cgmy}(t)=\exp\left\{\int_{\mathbb{R}}(e^{itu}-1)\nu_{cgmy}(du) \right\},~~t \in \mathbb{R},
			\end{equation}
			\noindent
			where the L\'evy measure $\nu_{cgmy}$ is

			\begin{equation}\label{CGMYlevy}
			\nu_{cgmy}(du)=\left(\frac{\alpha }{u^{1+\beta}}e^{-\lambda^{+}u}\mathbf{I}_{(0,\infty)}(u)+\frac{\alpha  }{|u|^{1+\beta}}e^{-\lambda^{-}|u|}\mathbf{I}_{(-\infty,0)}(u)\right)du,
			\end{equation}

			\noindent
			and the parameters $\alpha,\lambda^{+},\lambda^{-} \in (0,\infty)$ and $\beta \in [0,1)$. We denote it by $CGMY(\alpha,$ $ \beta,\lambda^{+},\lambda^{-}).$ Note that $CGMY(\alpha, \beta,\lambda^{+},\lambda^{-}) \overset{d}{=}GTSD(\mu,\beta,\alpha^{+},\lambda^{+},\alpha^{-})\overset{d}{=}IDD(\mu,0,\nu_{cgmy})$, where $\mu=\int_{\mathbb{R}}u\nu_{cgmy}(du)$. It is known \cite[Remark 7.11]{k13} that the $pdf$ of $CGMY$ distributions can not be expressed in closed form. We refer the reader to \cite{CGMY2} for more properties of $CGMY$ distributions.
		
		\section{A generalized Stein-type covariance identity}\label{mainr}
		\noindent
		In this section, we establish a covariance identity for infinitely divisible rvs. Before stating our result, let us define, for an integer $k\geq 1,$
		\begin{align}\label{qnnot1}
		\eta_{k}^{+}(u)=\int_{u}^{\infty}y^{k} \nu(dy), ~u>0,\text{ and }\eta_{k}^{-}(u)=-\int_{-\infty}^{u}y^{k} \nu(dy),~u<0,
		\end{align}
		\noindent
		where $\nu$ is the L\'evy measure of the $IDD$ (see \eqref{nkey1}). Let
		\begin{align}\label{qnnot0}
		\eta_{k} (u):= \eta_{k}^{+}(u)\mathbf{I}_{(0,\infty)}(u)+ \eta_{k}^{-}(u)\mathbf{I}_{(-\infty,0)}(u),~u\in \mathbb{R}.
		\end{align} 
		\noindent
		Let $X\sim IDD(\mu,0,\nu)$ with cf $\phi_{X}(t)$ given in \eqref{qchfgtsd}. Also let $(X_s,Y_s)$ be a infinitely divisible random vector with joint cf
		\begin{align}\label{jointcf}
		\phi_{s}(t,z)=\phi_X^{1-s}(t)\phi_X^{1-s}(z) \phi_X^s (t+z),	
		\end{align}
		\noindent
		for all $t,z \in \mathbb{R}$ and $s\in [0,1]$. It can be seen that $\phi_{X_s}(t)=\phi_{Y_s}(t)=\phi_{X}(t)$, so that $X_s$ and $Y_s$ are identically, but not independently, distributed.
 	
		\begin{thm}\label{qtheorem1}
			Let $X\sim\text{IDD}(\mu, 0, \nu)$ which has $\mathbb{E}(X^n) < \infty$, $n\in \mathbb{N}$, and the random vector $(X_s,Y_s)$ have the cf given in \eqref{jointcf}. Then 	
			\begin{equation}\label{qeqn1}
			\text{Cov}(X^{n},g(X))=\sum_{k=0}^{n-1}\binom{n}{k}\displaystyle\int_{0}^{1}\mathbb{E}\left(Y_s^{k}\displaystyle\int_{\mathbb{R}}g^{\prime}(X_s+v)\eta_{n-k}(v)dv\right)ds,
			\end{equation}
			where $g$ is an absolutely continuous function with $\mathbb{E}\left(Y_s^{k}\int_{\mathbb{R}}g^{\prime}(X_s+v)\eta_{n-k}(v)dv\right)<\infty$, for $0\leq k \leq (n-1)$.
		\end{thm}
		\begin{proof}
			Recall first (see \eqref{qchfgtsd}) that the cf of $X$ is 	
			\begin{align}\label{xchfgtsd}
			\phi_{X}(t)&=\exp \left(i t\mu+\int_{\mathbb{R}}\left(e^{itu} -1-itu\right)\nu(du)\right) ,~~t\in\mathbb{R}.
			\end{align}
			\noindent
			To prove our result, we use the covariance representation obtained in \cite[Proposition 2]{covrep} for infinitely divisible rvs. Let $ \mathcal{G}^{1}$ be the class of real-valued differentiable functions on $\mathbb{R}$ such that $f$ and $f^{\prime}$ are bounded on any bounded interval, the set of discontinuity points of $f^{\prime}$ has zero probability, and with $$\mathbb{E}\bigg(f^2(X)+\int_{\mathbb{R}} (f(X+u)-f(X)   )^{2} \nu (du)\bigg)<\infty.$$
			
			\noindent
			Then we have for $f,g \in \mathcal{G}^{1}$,	
			\begin{align}\label{newlab0}
			Cov(f(X),g(X))=\displaystyle\int_{0}^{1}\mathbb{E}\int_{\mathbb{R}} \bigg(f(Y_s +u)-f(Y_s) \bigg)\bigg( g(X_s +u) -g(X_s) \bigg)\nu(du)ds,
			\end{align} 
			\noindent
			where $(X_s, Y_s)$ has the cf given in \eqref{jointcf}. Let now $f(x)=x^{n}$ in \eqref{newlab0}, then we get
			\begin{align*}
			Cov(X^{n},g(X))=&\displaystyle\int_{0}^{1}\mathbb{E}\bigg( \displaystyle\int_{\mathbb{R}}\bigg( (Y_s+u)^{n} -Y_s^{n}   \bigg) \bigg(g(X_s+u)-g(X_s) \bigg)\nu(du)\bigg)ds\\
			=&\displaystyle\int_{0}^{1}\mathbb{E}\left(  \int_{\mathbb{R}}\left(\sum_{k=0}^{n-1} \binom{n}{k}Y_s^{k}u^{n-k} \right)\bigg(g(X_s+u)-g(X_s)\bigg)  \nu(du)    \right)ds\\
			=&\sum_{k=0}^{n-1} \binom{n}{k}\displaystyle\int_{0}^{1}\mathbb{E}\left(Y_s^{k}\displaystyle\int_{0}^{\infty}\bigg(g(X_s+u) -g(X_s)\bigg)u^{n-k}\nu (du)\right)ds\\
			&+\sum_{k=0}^{n-1} \binom{n}{k}\displaystyle\int_{0}^{1}\mathbb{E}\left(Y_s^{k}\displaystyle\int_{-\infty}^{0}\bigg(g(X_s+u) -g(X_s)\bigg)u^{n-k}\nu (du)\right)ds\\
			=&\sum_{k=0}^{n-1} \binom{n}{k}\displaystyle\int_{0}^{1}\mathbb{E}\left(Y_s^{k}\displaystyle\int_{0}^{\infty}\left( \int_{0}^{u}g^{\prime}(X_s+v)dv\right)u^{n-k}\nu (du)\right)ds\\
			&+\sum_{k=0}^{n-1} \binom{n}{k}\displaystyle\int_{0}^{1}\mathbb{E}\left(Y_s^{k}\displaystyle\int_{-\infty}^{0}\left( \int_{u}^{0}g^{\prime}(X_s+v)dv\right)(-u^{n-k})\nu (du)\right)ds\\
			=&\sum_{k=0}^{n-1}\binom{n}{k}\displaystyle\int_{0}^{1}\mathbb{E}\left(Y_s^{k}\displaystyle\int_{0}^{\infty}\bigg(g^{\prime}(X_s+v)\int_{v}^{\infty}u^{n-k}\nu(du)\bigg) dv\right)ds\\
			&+\sum_{k=0}^{n-1}\binom{n}{k}\displaystyle\int_{0}^{1}\mathbb{E}\left(Y_s^{k}\displaystyle\int_{-\infty}^{0}\bigg(g^{\prime}(X_s+v)\int_{-\infty}^{v}(-u^{n-k})\nu(du)\bigg) dv\right)ds\\
			=&\sum_{k=0}^{n-1}\binom{n}{k}\displaystyle\int_{0}^{1}\mathbb{E}\bigg(Y_s^{k}\displaystyle\int_{\mathbb{R}}g^{\prime}(X_s+v) \bigg(\eta _{n-k}^{+}(v)\mathbf{I}_{(0,\infty)}(v)\\
			&\quad\quad\quad+ \eta _{n-k}^{-}(v)\mathbf{I}_{(-\infty,0)}(v)  \bigg)dv\bigg)ds\\
			=&\sum_{k=0}^{n-1}\binom{n}{k}\displaystyle\int_{0}^{1}\mathbb{E}\left(Y_s^{k}\displaystyle\int_{\mathbb{R}}g^{\prime}(X_s+v)\eta_{n-k}(v)dv\right)ds,
			\end{align*}		
			\noindent
			which proves the result.
		\end{proof}
		\noindent
		We next show that, the Proposition 3.8 of \cite{arras0} follows for the case $n=1$.

		\vfill
		\begin{cor}\label{theorem1}
		Let $X\sim IDD(\mu,0,\nu)$ have moments upto second order and $Y_1$ (independent of $X$) has the density (see \cite[p.25]{arras0} and \eqref{cumul})
		\begin{align}\label{pdfzbias}
		f_{1}(y)=\frac{\eta_{1}(y)}{\int_{\mathbb{R}} y^{2} \nu (dy)}=\frac{\eta_{1}(y)}{Var(X)}, ~y\in\mathbb{R}.
		\end{align}
	 Then 	
	\begin{equation}\label{eqn1}
				\text{Cov}(X,g(X))=Var(X)\mathbb{E} \left(g^{\prime}(X+Y_1)\right),
			\end{equation}
			\noindent
			where $g$ is such that $\mathbb{E} \left(g^{\prime}(X+Y_1)\right)<\infty.$
		\end{cor}
	\begin{proof}
		From \eqref{qeqn1} and \eqref{cumul}, we have
		\begin{align*}
			Cov(X,g(X))&=\displaystyle\int_{0}^{1}\mathbb{E}\left(\displaystyle\int_{\mathbb{R}}g^{\prime}(X_s+v)\eta_{1}(v)dv\right)ds\\
			&=\displaystyle\int_{0}^{1}\mathbb{E}\left(\displaystyle\int_{\mathbb{R}}g^{\prime}(X+v)\eta_{1}(v)dv\right)ds \text{ (since $X_s \overset{d}{=}X$)}\\
			&=\mathbb{E}\left(\displaystyle\int_{\mathbb{R}}g^{\prime}(X+v)\eta_{1}(v)dv\right)\\
			&=Var(X)\mathbb{E} \left(g^{\prime}(X+Y_1)\right),
		\end{align*}
	\noindent
	which proves the result.	
	\end{proof}
	\begin{rem}
		(i) If $X\sim \mathcal{N}(0,\sigma^2),$ then Stein identity is $$\mathbb{E}(Xg(X))=\sigma^2 \mathbb{E} (g^{\prime} (X)), $$
		\noindent
		where $g$ is an absolutely continuous function with $\mathbb{E}[g^{\prime}(X)] < \infty$. However, in general if $X$ is such that $\mathbb{E}(X)=0,$ and $Var(X)=\sigma^2$, then $X^*$ is said to have $X$-zero bias distribution if 
		$$\mathbb{E}(Xg(X))=\sigma^2 \mathbb{E} (g^{\prime} (X^*)),$$
		\noindent
		for all absolutely continuous $``g"$.
		
		\item (ii)Let $X^*=X+Y_1$. Then \eqref{eqn1} becomes
		$$Cov(X,g(X))=Var(X)\mathbb{E}\left(g^{\prime} (X^*) \right), $$
		and so we may say $X^*$ has the $X$-non-zero bias distribution. Usually, the covariance identities are for $X$-zero bias distributions, see \cite{cov1}.
	\end{rem}
	
	\noindent
	Note that
	\begin{align}
	\nonumber	\displaystyle\int_{0}^{\infty}\eta_{k}^{+}(u)du&=\int_{0}^{\infty}\int_{u}^{\infty}y^{k}\nu(dy)du,~u>0\\
	\nonumber	&=\int_{0}^{\infty}y^{k}\bigg(\int_{0}^{y}du \bigg)\nu(dy)\\
		&=\int_{0}^{\infty}y^{k+1}\nu(dy).\label{numeradensity}
	\end{align}
	\noindent
	If $X\sim IDD (\mu, 0, \nu)$ is non-negative, then the $k$-th cumulant of $X$ is $C_{k}(X)=\int_{0}^{\infty}u^{k}\nu(du),$ for some integer $k\geq 1,$ and in that case

	\begin{align}\label{qpdfzbias}
		f_{k}(y)=\frac{\eta_{k}^{+}(y)}{C_{k+1}(X)}, ~y \in (0,\infty),
	\end{align}
	\noindent
	is a density on $\mathbb{R}_{+}=(0,\infty)$.
		
	\vskip 2ex
	\noindent
	We now have the following Corollary.	
	
	\begin{cor}\label{pcid}
		Let $X\sim\text{IDD}(\mu, 0, \nu)$ be non-negative which has $\mathbb{E}(X^n) <\infty$, $n\in \mathbb{N}$, and $Y_{k}$ ($k\geq 1$), independent of $X$, has the density as defined in \eqref{qpdfzbias}. Also let the random vector $(X_s,Y_s)$ have the cf given in \eqref{jointcf}. Then 	
		\begin{equation}
		\text{Cov}(X^{n},g(X))=\sum_{k=0}^{n-1}\binom{n}{k}C_{n-k+1}(X) \displaystyle\int_{0}^{1}\mathbb{E} \left(Y_s^{k}g^{\prime}(X_s+Y_{n-k})\right)ds,
		\end{equation}
		where $g$ is such that $\mathbb{E} \left(Y_s^kg^{\prime}(X_s+Y_{n-k})\right)<\infty,$ for $0\leq k \leq (n-1).$ 
	\end{cor}
	\begin{proof}
		Since $X\sim IDD(\mu, 0, \nu)$ is non-negative, the support of $\nu$ is in $(0,\infty)$. Also, from

		\noindent
		the assumption $\mathbb{E}(X^n) <\infty$, $n\in \mathbb{N}$, for all $v>0$, $\eta_k^{+}(v)=\int_{v}^{\infty}y^k \nu (dy), ~k\geq1.$ Hence from Theorem \ref{qtheorem1}, we get
		\begin{align*}
		Cov(X^{n},g(X))
		=&\sum_{k=0}^{n-1}\binom{n}{k}\displaystyle\int_{0}^{1}\mathbb{E}\left(Y_s^{k}\displaystyle\int_{0}^{\infty}g^{\prime}(X_s+v)\eta_{n-k}^{+}(v)dv\right)ds\\
		=&\sum_{k=0}^{n-1}\binom{n}{k}C_{n-k+1}(X) \displaystyle\int_{0}^{1}\mathbb{E} \left(Y_s^{k}g^{\prime}(X_s+Y_{n-k})\right)ds,
		\end{align*}		
		\noindent
		which proves the result.
	\end{proof}		
\begin{rem}	

 (i) In Theorem \ref{qtheorem1} and Corollary \ref{theorem1}, we assume that the rv $X$ has higher order moments, which are satisfied for Poisson, negative binomial, gamma, Laplace, bilateral-gamma and variance-gamma distributions.
	
\item[(ii)] Note that the identity \eqref{eqn1} does not hold for stable distributions which do not have finite second moment.  
\end{rem}		
\noindent
Next, we discuss some examples.
\begin{exmp}(Generalized tempered stable distributions)
	Let $X \sim GTSD(\mu,\beta,\alpha^{+}$ $,\lambda^{+},\alpha^{-},\lambda^{-})$ with cf $\phi_{\text{gts}}$ given in \eqref{nep5}. By Corollary \ref{theorem1}, a covariance identity for $X$ is
	\begin{align}
	Cov(X,g(X))=\Gamma(2-\beta)\left(\frac{\alpha^{+}}{(\lambda^{+})^{2-\beta}}+  \frac{\alpha^{-}}{(\lambda^{-})^{2-\beta}} \right)\mathbb{E}\left(g^{\prime}(X+Y_1) \right),
	\end{align}
	\noindent
	where the rv $Y_1$ has the density
	\begin{align}\label{newden00}
	f_1(y)= \frac{\alpha^+ \left[\int_{y}^{\infty}u^{-\beta}e^{-\lambda^+ u}du\right]\textbf{I}_{(0,\infty)}(y)+\alpha^- \left[\int_{-\infty}^{y}|u|^{-\beta}e^{-\lambda^- |u|}du\right]\textbf{I}_{(-\infty,0)}(y)   }{\Gamma(2-\beta)\left(\frac{\alpha^{+}}{(\lambda^{+})^{2-\beta}}+  \frac{\alpha^{-}}{(\lambda^{-})^{2-\beta}} \right)},~y\in\mathbb{R}.
	\end{align}
	\vfill

	\noindent
	Also, by Theorem \ref{qtheorem1}, a covariance identity for $X^{2}$ is
	\begin{align*}
	Cov(X^{2},g(X))=&\sum_{k=0}^{1}\binom{2}{k}\displaystyle\int_{0}^{1}\mathbb{E}\left(Y_s^{k}\displaystyle\int_{\mathbb{R}}g^{\prime}(X_s+v)\eta_{2-k}(v)dv\right)ds\\
	=&2\Gamma(2-\beta)\left(\frac{\alpha^{+}}{(\lambda^{+})^{2-\beta}}+  \frac{\alpha^{-}}{(\lambda^{-})^{2-\beta}} \right)\displaystyle\int_{0}^{1}\mathbb{E}\bigg(Y_sg^{\prime}(X_s+Y_1)\bigg)ds\\
	&\quad\quad+\mathbb{E}\left(\displaystyle\int_{\mathbb{R}}g^{\prime}(X+v)\eta_{2}(v)dv \right)~(\text{since } X_s \overset{d}{=}X),
	\end{align*}
	\noindent
	where $(X_s,Y_s)$ has the cf $\phi_s(t,z)=\phi_{\text{gts}}^{1-s}(t)\phi_{\text{gts}}^{1-s}(z) \phi_{\text{gts}}^s (t+z),~t,z\in\mathbb{R},$ $Y_1$ has the density as in \eqref{newden00} and $\eta_{2}$ is given by
	\begin{align*}
	\eta_{2}(v)&= \alpha^{+}\left[\int_{v}^{\infty}u^{-(\beta -1)}e^{-\lambda^{+}u}du\right]\mathbf{I}_{(0,\infty)}(v)\\
	&\quad\quad-\alpha^{-}\left[\int_{-\infty}^{v}|u|^{-(\beta -1)}e^{-\lambda^{-}|u|}du   \right]\mathbf{I}_{(-\infty,0)}(v), ~v\in\mathbb{R}.
	\end{align*}
\end{exmp}

\begin{exmp}(Variance-gamma distributions)
	Let $X\sim VGD_0(\mu_{0},\alpha,\lambda^{+},\lambda^{-})$ with cf $\phi_{vg}$ given in \eqref{cfvgd1}. By Corollary \ref{theorem1}, a covariance identity for $X$ is
	\begin{align}\label{VGDCI}
	Cov (X,g(X))=\frac{\alpha((\lambda^{+})^{2}+(\lambda^{-})^{2})}{(\lambda^{+})^{2}(\lambda^{-})^{2}} \mathbb{E}\left(g^{\prime}(X+Y_1)\right),
	\end{align}
	\noindent
	where the rv $Y_1$ has the density	
	\begin{align}
	 f_{1}(y)=\frac{\eta_{1}(y)}{\int_{\mathbb{R}} y^{2} \nu_{vg} (dy)}
	=\frac{(\lambda^{+})^{2}(\lambda^{-})^{2}}{(\lambda^{+})^{2}+(\lambda^{-})^{2}}\left(\frac{e^{-\lambda^{+}y}}{\lambda^{+}}\mathbf{I}_{(0,\infty)}(y)+ \frac{e^{\lambda^{-}y}}{\lambda^{-}}\mathbf{I}_{(-\infty,0)}(y)  \right).\label{VGCIDEN}
	\end{align}

	\noindent
	Also, by Theorem \ref{qtheorem1}, a covariance identity for $X^{2}$ is
		\begin{align*}
	Cov(X^{2},g(X))=&\sum_{k=0}^{1}\binom{2}{k}\displaystyle\int_{0}^{1}\mathbb{E}\left(Y_s^{k}\displaystyle\int_{\mathbb{R}}g^{\prime}(X_s+v)\eta_{2-k}(v)dv\right)ds\\
	=&\frac{2\alpha\bigg((\lambda^{+})^{2}+(\lambda^{-})^{2}\bigg)}{(\lambda^{+})^{2}(\lambda^{-})^{2}}\displaystyle\int_{0}^{1}\mathbb{E}\bigg(Y_sg^{\prime}(X_s+Y_1)\bigg)ds\\
	&\quad\quad+\mathbb{E}\left(\displaystyle\int_{\mathbb{R}}g^{\prime}(X+v)\eta_{2}(v)dv \right)~(\text{since } X_s \overset{d}{=}X),
	\end{align*}
	\noindent
	where $(X_s,Y_s)$ has the cf $\phi_s(t,z)=\phi_{vg}^{1-s}(t)\phi_{vg}^{1-s}(z) \phi_{vg}^s (t+z),~t,z\in\mathbb{R},$ $Y_1$ has the density as in \eqref{VGCIDEN} and $\eta_{2}$ is given by
	\begin{align*}
	\eta_{2}(v)= \alpha^{+}\left[\int_{v}^{\infty}ue^{-\lambda^{+}u}du\right]\mathbf{I}_{(0,\infty)}(v)-\alpha^{-}\left[\int_{-\infty}^{v}|u|e^{-\lambda^{-}|u|}du   \right]\mathbf{I}_{(-\infty,0)}(v),~v\in\mathbb{R}.
	\end{align*}
		
	\noindent
	When $\alpha=1,$ $\lambda^{+}=\lambda^{-}=\frac{1}{\delta},$ $\delta>0$, we get the L\'evy measure for $La(\mu_0,\delta^{2})$ as \begin{align}\label{laplevy}
	\nu_{l}(du)=\left(\frac{e^{-\frac{u}{\delta}}}{u}\textbf{I}_{(0,\infty) }(u)+\frac{e^{-\frac{|u|}{\delta}}}{|u|}\textbf{I}_{(-\infty,0) }(u) \right)du,
	\end{align}
	\noindent
	and a covariance identity for a rv $X\sim La (\mu_0, \delta^{2})$ is

	\begin{align*}
	Cov(X,g(X))=2\delta^{2}\mathbb{E}\left(g^{\prime}(X+Y_1)\right).
	\end{align*}

	\noindent
	Here the rv $Y_1$ has density
	\begin{align}
	\nonumber f_1(y)&=\frac{\eta_{1}(y)}{\int_{\mathbb{R}}u^{2}\nu_{l}(du)  }\\
	\nonumber&=\frac{1}{2 \delta^{2}}\bigg( \left[\int_{y}^{\infty}u \nu_{l}(du) \right]\mathbf{I}_{(0,\infty)}(y)+ \left[\int_{-\infty}^{y}(-u) \nu_{l}(du) \right]\mathbf{I}_{(-\infty,0)}(y)   \bigg) \\
\nonumber	&=\frac{1}{2 \delta}\bigg( e^{-\frac{y}{\delta}} \mathbf{I}_{(0,\infty)}(y)+ e^{\frac{y}{\delta}}\mathbf{I}_{(-\infty,0)}(y)   \bigg)\\
\label{laplevy100}	&=\frac{1}{2 \delta}e^{- \frac{|y|}{\delta}},
	\end{align}
	
\noindent
which is $La(0,\delta^2)$ density.

\end{exmp}		

\begin{exmp} (Poisson distribution)
	Let $X \sim Poi (\lambda)$. Then by Corollary \ref{theorem1}, a covariance identity for $X$ is
	\begin{align}
	Cov(X,g(X))=\lambda \mathbb{E}\left(g^{\prime}(X+Y_1)\right).
	\end{align}
	\noindent
	If $X\sim Poi(\lambda)$, then we have $\nu(du)=\lambda \delta_{1}(du)$, and
	 $$\eta_{1}(y)=\int_{y}^{\infty}u\nu(du)=\lambda\int_{y}^{\infty}u\delta_{1}(du)=\lambda \textbf{I}_{(0,1)}(y),$$
	 \noindent
	  $\text{and } \int_{0}^{\infty}u^2 \nu(du)=\lambda.$ So, the rv $Y_1$ has the density
	$$f_1(y)=\frac{\eta_{1}(y)}{\int_{\mathbb{R}}u^2 \nu(du)}=\textbf{I}_{(0,1)}(y),~y>0.$$
	\noindent
	That is, $Y_1$ follows uniform $U(0,1)$ distribution.
	
	\noindent
	Also, by Corollary \ref{pcid}, a covariance identity for $X^{n},$ $n\geq2$, is
	\begin{align}
\nonumber	Cov(X^{n},g(X))&=\sum_{k=0}^{n-1}\binom{n}{k}C_{n-k+1}(X)\displaystyle\int_{0}^{1}\mathbb{E} \left(Y_s^{k}g^{\prime}(X_s+Y_{n-k}) \right)ds\\
	&=\lambda \sum_{k=0}^{n-1}\binom{n}{k}\displaystyle\int_{0}^{1}\mathbb{E} \left(Y_s^{k}g^{\prime}(X_s+Y_{n-k}) \right)ds,\label{poiidnew}
	\end{align}
	\noindent
	where the last step follows as $C_{k}(X)=\lambda,$ for all $k\geq 1$, and $Y_{k}$'s $(k\geq 1)$ have the density
\begin{align*}
	f_{k}(y) =\frac{\int_{y}^{\infty}u^{k}\nu(du) }{C_{k+1}(X)}=  \int_{y}^{\infty}u^{k}\delta_{1}(du) =\textbf{I}_{(0,1)}(y),~y>0,   
	\end{align*}
\noindent
which is uniform $U(0,1)$ density. Note also that $(X_s,Y_s)$ has the cf
\begin{align}\label{poiidnew2}
\phi_s(t,z)=\exp \bigg(\lambda(1-s)(e^{it}-1) +\lambda(1-s)(e^{iz}-1)+\lambda s(e^{i(t+z)}-1)   \bigg),~t,z\in\mathbb{R},
\end{align}	
since $\phi_{X}(t)=\exp\left(\lambda (e^{it} -1)  \right)$. 

\noindent
For example, when $n=2$ and $g(x)=x^2$, then from \eqref{poiidnew}, we get
\begin{align}
\nonumber	Cov(X^2,X^2)&=2\lambda\sum_{k=0}^{1}\binom{2}{k}\displaystyle\int_{0}^{1}\mathbb{E}\bigg(Y_{s}^{k} \bigg(X_s + Y_{2-k}   \bigg)   \bigg)ds\\
\nonumber	&=2\lambda \displaystyle\int_{0}^{1}\mathbb{E}\bigg(X_s+Y_2 \bigg)ds+ 4\lambda \displaystyle\int_{0}^{1}\mathbb{E}\bigg(Y_s\bigg(X_s+Y_1 \bigg)\bigg)ds\\
\nonumber	&=2\lambda\displaystyle\int_{0}^{1}\bigg(\lambda +\frac{1}{2} \bigg)ds+ 4\lambda  \displaystyle\int_{0}^{1}\bigg(\mathbb{E}( X_s Y_s)+\mathbb{E}(X_s ) \mathbb{E}(Y_1)  \bigg) ds    \\
	&=4\lambda  \displaystyle\int_{0}^{1}\mathbb{E}\bigg( X_s Y_s\bigg)  ds   +\lambda (4\lambda +1).\label{poiidnew1}
\end{align}

\noindent
From \eqref{poiidnew2}, we get
\begin{align}\label{poiidnew3}
	\mathbb{E}\bigg(X_sY_s \bigg)=(i)^{-2}\frac{\partial^{2}}{\partial t \partial z} \phi_{s}(t,z) \bigg|_{t=z=0}=\lambda^{2}+\lambda s. 
\end{align}	
\noindent
Using \eqref{poiidnew3} in \eqref{poiidnew1}, we get
\begin{align}
\nonumber Cov(X^2,X^2)&=4\lambda \displaystyle\int_{0}^{1}\bigg(\lambda^{2} +\lambda s \bigg)ds +\lambda (4\lambda +1)\\
&=\lambda (4\lambda^2 + 6\lambda +1),	
\end{align}
\noindent
which can be checked through direct calculation also.	 
\end{exmp}

\section{Variance bounds}\label{vbbd}
\noindent
In this section, we discuss the use of covariance identities to establish bounds on variance of function of infinitely divisible rv.		
\subsection{Cacoullos type variance bounds}
Deriving upper bounds on the variance of a function of a rv has a long and rich history, starting in 1981 from the work of Chernoff \cite{Chernoff}. Over the years, upper and lower variance bounds have received much interest in the statistics literature; see, for instance \cite{caco,cov1} and the references therein. In the following result, we obtain two-side bounds (Cacoullos type; see\cite{caco}) on the variance of $g(X)$, where $X\sim IDD(\mu,0,\nu)$. Our approach is new and exploits the identity \eqref{eqn1}, which is based on non-zero-biased distribution of $X$. Recall first that an upper bound for variance of function of an infinitely divisible rv due to Chen \cite[Theorem 4.1]{lhy} is as follows.
\begin{lem}
	Let $X\sim\text{IDD}(\mu, 0, \nu)$. Then 
	\begin{align}\label{Hcorep1}
		Var(g(X)) \leq \mathbb{E}\displaystyle\int_{\mathbb{R}} \left(g(X+u)-g(X)   \right)^{2} \nu (du),  
	\end{align}
	where $g:\mathbb{R} \to \mathbb{R}$ is such that $\mathbb{E}\bigg(g^2(X)+\int_{\mathbb{R}} \left(g(X+u)-g(X)   \right)^{2} \nu (du)\bigg)<\infty.$
\end{lem}

\noindent
Next, we establish a Cacoullos type variance bounds for an infinitely divisible rv.
\vfill

\begin{thm}\label{AVBth1}		
	Let $X\sim IDD(\mu,0,\nu)$ and $Y_1$ have density defined in \eqref{pdfzbias}. Then
	\begin{align}\label{VBGTSD}
	Var (X)\mathbb{E}^{2}\left[ g^{\prime} (X+Y_1)\right]\leq Var (g(X)) \leq Var(X) \mathbb{E}\left[ g^{\prime}(X+Y_1)\right]^{2},
	\end{align}
	where $g$ is an absolutely continuous function with $\mathbb{E} \left(g^{\prime}(X+Y_1)\right)<\infty.$
\end{thm}
\begin{proof}
	From \eqref{eqn1}, we have
	\begin{align}
	\nonumber (Var(X))^{2}\mathbb{E}^{2}\left[ g^{\prime} (X+Y_1)\right]=&(C_{2}(X))^{2}\mathbb{E}^{2}\left[ g^{\prime} (X+Y_1)\right]\\
	\nonumber=&\left[ Cov(X,g(X)) \right]^{2}\\
	\nonumber	=&\left[ \mathbb{E}\left( (X-\mathbb{E}X)(g(X)-\mathbb{E}g(X)\right)\right]^{2}\\
	\label{VB0} & \leq Var(X)Var(g(X)),
	\end{align}
	\noindent
	by Cauchy-Schwarz inequality. That is,
	\begin{align}\label{VB1}
	Var(g(X))\geq Var(X) \mathbb{E}^{2}\left[ g^{\prime} (X+Y_1)\right],
	\end{align}
	\noindent
	where $Y_1$ has the density given in \eqref{pdfzbias}.
	
	\noindent
	Next, observe that
	\begin{align}
	\nonumber	Var(X)\mathbb{E}\left[g^{\prime}(X+Y_1)\right]^{2}=&C_{2}(X)\mathbb{E}\left[g^{\prime}(X+Y_1)\right]^{2}\\
	\nonumber=&C_{2}(X)\mathbb{E}\int_{\mathbb{R}} \left[g^{\prime}(X+v)\right]^{2} f_{1}(v)dv\\
	\nonumber=&\mathbb{E}\displaystyle\int_{\mathbb{R}}\left[g^{\prime}(X+v)\right]^{2} \eta_{1}(v)dv \\
	\nonumber=&\mathbb{E}\left[\displaystyle\int_{0}^{\infty}\left( g^{\prime}(X+v) \right)^{2}\int_{v}^{\infty}u\nu(du) dv\right]\\
	\nonumber&+\mathbb{E}\left[\displaystyle\int_{-\infty}^{0}\left( g^{\prime}(X+v) \right)^{2}\int_{-\infty}^{v}(-u)\nu(du) dv   \right]\\
	\nonumber=&\mathbb{E}\left[\displaystyle\int_{0}^{\infty}\left(\int_{0}^{u}\left( g^{\prime}(X+v) \right)^{2}dv\right) u\nu(du) \right]\\
	\nonumber&+\mathbb{E}\left[\displaystyle\int_{-\infty}^{0}\left(\int_{u}^{0}\left( g^{\prime}(X+v) \right)^{2}dv\right) (-u)\nu(du) \right]\\
	\nonumber=&\mathbb{E}\displaystyle\int_{\mathbb{R}}\left(\int_{0}^{u}\left(g^{\prime} (X+v) \right)^{2} dv\right) u \nu(du)\\
	\nonumber=& \mathbb{E}\displaystyle\int_{\mathbb{R}} \left(\int_{0}^{u}dv\int_{0}^{u}\left(g^{\prime} (X+v) \right)^{2} dv\right)  \nu(du) \\
	\nonumber\geq& \mathbb{E}\displaystyle\int_{\mathbb{R}} \left(\int_{0}^{u}g^{\prime}(X+v)dv  \right)^{2}\nu(du)\\
	\nonumber&~~~~~~~~~~~~~~~~~~~~~~~~~~\quad\quad\quad\text{(by Cauchy-Schwarz inequality)}\\
	\nonumber=&\mathbb{E}\displaystyle\int_{\mathbb{R}}(g(X+u) -g(X) )^{2} \nu(du)\\
	\quad\quad \quad\quad \geq &Var (g(X))\label{VB2} \text{ (using \eqref{Hcorep1})}.
	\end{align}
	\noindent
	 Hence combining \eqref{VB1} and \eqref{VB2}, the desired conclusion follows.	
\end{proof}

\noindent
Some examples follow.
\begin{exmp}(Gamma distribution)
	Let $X\sim Ga(a,b)$, ($a,b>0$) with $pdf$
	\begin{align}
		f_{X}(x)=\frac{b^{a}}{\Gamma(a)}x^{a-1}e^{-bx},~x>0.
	\end{align}
	\noindent 
	Then $X\sim IDD(\mu,0,\nu_g)$, where $\nu_{g}(du)=a\frac{e^{-bu}}{u}\mathbf{I}_{(0,\infty)}(u)du$, $\mu=\int_{\mathbb{R}}u\nu_{g}(du)$ and $Y_1 \sim Ga(1,b)$. Hence, by \eqref{VBGTSD}, we get
	\begin{align}\label{gammavarbd}
		\frac{a}{b^{2}} \mathbb{E}^{2}\left(g^{\prime}(Z)\right) \leq Var(g(X)) \leq \frac{a}{b^{2}} \mathbb{E}\left( g^{\prime} (Z)\right)^{2},
	\end{align}
	where the rv $Z\sim Ga(a+1,b)$.
	
	\vskip 1ex
	\noindent		
    Observe that, the result in \eqref{gammavarbd} coincides with the one given in \cite[equ. (4.1)]{geoP}. 	
\end{exmp}

\noindent
The next two examples are not discussed in the literature.
\begin{exmp}(Laplace distribution)
Let $X \sim La(0,\delta^{2})$, $\delta>0$ with $pdf$

\begin{align*}
	f_{X}(x)=\frac{1}{2 \delta}e^{-\frac{|x|}{\delta}},~x\in\mathbb{R}.
\end{align*}
\noindent 
Then $X\sim IDD(\mu,0,\nu_l)$, where $\mu=\int_{\mathbb{R}}u\nu_{l}(du)$, $\nu_l$ is given in \eqref{laplevy}, and $Y_1$ has the density given in \eqref{laplevy100}.

\noindent
 Hence, by \eqref{VBGTSD}, we get
\begin{align}
2\delta^{2} \mathbb{E}^{2}\left(g^{\prime}(Z)\right) \leq Var(g(X)) \leq  2\delta^{2}\mathbb{E}\left( g^{\prime} (Z)\right)^{2},
\end{align}
\noindent
where the rv $Z\sim VGD_{1}(0,\delta^{2},4,0)$.
\end{exmp}
\begin{exmp}(Two-sided exponential distribution)
	Let $X$ have a two-sided exponential distribution with parameters $a>0$ and $b>0$ with density

	$$f_{X}(x)=\frac{ab}{a+b}\left(e^{-ax}\textbf{I}_{(0,\infty)}(x)+e^{bx}\textbf{I}_{(-\infty,0)}(x)  \right),~x\in \mathbb{R}.$$
\noindent 
Then $X\sim IDD(\mu,0,\nu_e)$, where 
$$\nu_e(du)=\left( \frac{e^{-au}}{u}\mathbf{I}_{(0,\infty)}(u)-\frac{e^{bu}}{u}\mathbf{I}_{(-\infty,0)}(u) \right),$$ $\mu=\int_{\mathbb{R}}u\nu_{e}(du)$ and $Y_1$ has density

$$f_1(y)=\frac{a^{2}b^{2}}{a^{2}+b^{2}}\left( \frac{e^{-ay}}{a}\mathbf{I}_{(0,\infty)}(y)+ \frac{e^{by}}{b}\mathbf{I}_{(-\infty,0)}(y)\right).$$

\noindent
Hence, by \eqref{VBGTSD}, we get
\begin{align}
\frac{a^{2}+b^{2}}{a^{2}b^{2}} \mathbb{E}^{2}\left(g^{\prime}(Z)\right) \leq Var(g(X)) \leq  \frac{a^{2}+b^{2}}{a^{2}b^{2}}\mathbb{E}\left( g^{\prime} (Z)\right)^{2},
\end{align}
where the rv $Z$ has density
$$f_{Z}(z)=\frac{a^{3}b^{3}}{(a+b)(a^{2}+b^{2})} \left(\frac{ze^{-az}}{a}\textbf{I}_{(0,\infty)}(z)+\frac{ze^{bz}}{b}\textbf{I}_{(-\infty,0)}(z)  \right).$$	 
\end{exmp}
	
	\subsection{Application to posterior distributions}\label{application}
	\noindent
	Here, we demonstrate, using Cacoullos type bounds, through examples to compute variance bounds of the parameters of the posterior distributions. Recently Daly {\it et al.} \cite[Example 2.3]{cov1} discussed variance bounds within a Bayesian context for the Pearson family by using Stein kernels. In the Bayesian methodology, the parameters are treated as random variables. Let $(X_1,\ldots,X_n)$ be a random sample of size $n$, where the joint density of $\Theta$ and $X_i$, $1\leq i \leq n,$ is $\pi(\theta,x)$. Assume the prior density of $\Theta$ is $\pi (\theta).$ Then we update from prior to posterior density as $\pi(\theta | x)=\kappa (x) \pi (\theta,x) \pi (\theta),$ where $\kappa(x)$ is a normalizing constant that depends only on the data $x=(x_1,\ldots,x_n)$.



	\begin{exmp}[Gamma data, inference on scale, gamma prior]
		Let $(X|\theta)\sim Ga(k,\theta)$ with $k,\theta>0,$ and $\Theta \sim Ga(a,b)$ with $a,b>0$, the prior distribution. Then the posterior distribution is $(\Theta |x) \sim Ga(nk+a,n\bar{x}+b)$, where $x=(x_1,\ldots,x_n)$. Hence, from \eqref{gammavarbd}, we get
		\begin{align}\label{pobd1}
		\frac{(nk+a)}{(n\bar{x}+b)^{2}} \mathbb{E}^{2}\left(g^{\prime}(Z)\right) \leq Var(g(\Theta|x)) \leq \frac{(nk+a)}{(n\bar{x}+b)^{2}} \mathbb{E}\left( g^{\prime} (Z)\right)^{2},
		\end{align}
		where the rv $Z\sim Ga(nk+a+1,n\bar{x}+b)$.
		
		\vskip 1ex
		\noindent
		For example, when $g(x)=x$, from \eqref{pobd1}, we get
		\begin{align}\label{pobd2}
			Var(\Theta |x)=\frac{(nk+a)}{(n\bar{x}+b)^{2}},
		\end{align}
		\noindent
		as expected.
	
	\noindent
	When $g(x)=x^2$, and using \eqref{pobd1}, we get
	\begin{align}\label{pobd3}
		 \frac{4(nk+a)(nk+a+1)^{2}}{(n\bar{x} +b)^{4}} \leq Var(g(\Theta|x )) \leq \frac{4(nk+a)(nk+a+1) (nk+a+2)}{(n\bar{x} +b)^{4}}.
	\end{align}
	\noindent
	Observe also that, the upper and lower bounds in \eqref{pobd3} coincides with the one given in \cite[Example B.4]{cov1}.	
	\end{exmp}
	
	\begin{exmp}[Poisson data, inference on scale, gamma prior] Let $(X|\theta)\sim Poi (\theta)$ with $\theta>0$, and $\Theta \sim Ga(a,b)$, with $a,b>0$, the prior distribution. Then the posterior distribution is $(\Theta|x) \sim Ga (n\bar{x}+a,n+b)$. Again from \eqref{gammavarbd}, we get
		\begin{align}\label{pobd4}
		\frac{(n\bar{x}+a)}{(n+b)^{2}} \mathbb{E}^{2}\left(g^{\prime}(Z)\right) \leq Var(g(\Theta|x)) \leq \frac{(n\bar{x}+a)}{(n+b)^{2}} \mathbb{E}\left( g^{\prime} (Z)\right)^{2},
		\end{align}
		where the rv $Z\sim Ga(n\bar{x}+a+1,n+b)$.

		\vskip 1ex
		\noindent
		For example, when $g(x)=x$, from \eqref{pobd4}, we get
		\begin{align}\label{pobd5}
		Var(\Theta |x)=\frac{(n\bar{x}+a)}{(n+b)^{2}},
		\end{align}
		\noindent
		as expected.
		
		\noindent
		When $g(x)=x^2$, we get from \eqref{pobd4}
		\begin{align}\label{pobd6}
		\frac{4(n\bar{x}+a)(n\bar{x}+a+1)^{2}}{(n +b)^{4}} \leq Var(g(\Theta|x )) \leq \frac{4(n\bar{x}+a)(n\bar{x}+a+1) (n\bar{x}+a+2)}{(n +b)^{4}}.
		\end{align}
		\noindent
		The bounds given above coincide with the one given in \cite[Example B.6]{cov1}.	
	\end{exmp}

\section{Applications}\label{appofre}	
\noindent
In this section, we discuss applications of our results.
\subsection{Stein identities via covariance identities}
Here, we demonstrate the use of covariance identities to provide Stein identities for several probability distributions. It is well known that such identities are useful to obtain explicit approximations, using Stein's method, for sums of independent rvs, but such applications are beyond the scope of this paper. Before stating our next results, we recall the Schwartz space of functions. Let $\mathcal{S}(\mathbb{R})$ be the Schwartz space defined by
$$\mathcal{S}(\mathbb{R}):=\left\{f\in C^\infty(\mathbb{R}): \lim_{|x|\rightarrow \infty} |x^m f^{(n)}(x)|=0, \text{ for all } m,n\in \mathbb{N}_{0}\right\},$$
\noindent
where $\mathbb{N}_{0}=\mathbb{N}\cup \{0\}$, $f^{(n)}$ denotes the $n$-th derivative of $f$ and $C^\infty(\mathbb{R})$ is the class of infinitely differentiable functions on $\mathbb{R}$ (see \cite{SteinSchwarz} for more details). This function space is useful for choosing an appropriate function space for Stein identities. Our first example yields a Stein identity for $CGMY$ distributions.

\begin{exmp}($CGMY$ distributions)\label{cgmy100}
	Let $X \sim CGMY (\alpha,\beta,\lambda^{+},\lambda^{-})$, whose L\'evy measure $\nu_{cgmy}$ is given in \eqref{CGMYlevy}.
	\noindent
	Recall that $CGMY (\alpha,\beta,\lambda^{+},\lambda^{-}) \overset{d}{=}IDD(\mu,0,\nu_{cgmy})$, where $\mu=\int_{\mathbb{R}}u \nu_{cgmy}(du).$ 
	Hence by Corollary \ref{theorem1}, we have
	\begin{align*}
		Cov(X,g(X))
		=Var(X)\mathbb{E}\bigg(g^{\prime} (X+Y_1)\bigg),
	\end{align*}
	\noindent
	where $Y_1$ has density
	$$f_1(v)=\frac{[\int_{v}^{\infty}u\nu_{cgmy}(du)]\mathbf{I}_{(0,\infty)}(v) + [\int_{-\infty}^{v}(-u)\nu_{cgmy}(du)]\mathbf{I}_{(-\infty,0)}(v)}{Var(X)},~v\in \mathbb{R}.$$
	
	\noindent
	So,
	\begin{align}
	\nonumber	\mathbb{E}\left(X-\mu\right)g(X) =&Cov(X,g(X))\\
	=&Var(X)\mathbb{E}\bigg(g^{\prime} (X+Y_1) \bigg)\\
	\nonumber=&\mathbb{E}\displaystyle\int_{0}^{\infty}g^{\prime}(X+v)\int_{v}^{\infty}u\nu_{cgmy}(du) dv\\
	\nonumber&+\mathbb{E}\displaystyle\int_{-\infty}^{0}g^{\prime}(X+v)\int_{-\infty}^{v}(-u)\nu_{cgmy}(du) dv\\
	\nonumber=&\mathbb{E}\displaystyle\int_{0}^{\infty}\left( \int_{0}^{u}g^{\prime}(X+v)dv\right)u\nu_{cgmy} (du)\\
	\nonumber&+\mathbb{E}\displaystyle\int_{-\infty}^{0}\left( \int_{u}^{0}g^{\prime}(X+v)dv\right)(-u)\nu_{cgmy} (du)\\
	\nonumber=&\mathbb{E}\displaystyle\int_{0}^{\infty}\left(g(X+u) -g(X)\right)u\nu_{cgmy} (du)\\
	\nonumber&+\mathbb{E}\displaystyle\int_{-\infty}^{0}\left(g(X+u) -g(X)\right)u\nu_{cgmy} (du)\\
	\label{cgmysid0}=&\mathbb{E}\left(  \int_{\mathbb{R}}u(g(X+u)-g(X))\nu_{cgmy}(du)    \right).
	\end{align}
	\noindent
	Since $\mu=\int_{\mathbb{R}}u \nu_{cgmy}(du),$ then \eqref{cgmysid0} simplifies to
	\begin{align}
	\mathbb{E} \left(Xg(X)- \int_{\mathbb{R}}ug(X+u)\nu_{cgmy}(du)    \right)=0,\label{num01}
	\end{align}
	
	\noindent
	which is a Stein identity for $CGMY (\alpha,\beta,\lambda^{+},\lambda^{-})$ distribution.
\end{exmp}

\noindent
Next, we obtain a Stein identity for $VGD$, using our approach.
\begin{exmp}(Variance-gamma distributions)\label{corvg}
	Let $X\sim VGD_0(\mu_{0},\alpha,\lambda^{+},\lambda^{-})$ whose L\'evy measure $\nu_{vg}$ is defined in \eqref{levyvgd}. Let $g\in \mathcal{S}(\mathbb{R})$. Now, using the fact $\mathbb{E}(X)=\mu_{0}+\int_{\mathbb{R}}u\nu_{vg}(du)$, and following steps similar to Example \ref{cgmy100}, it can be shown that
	\begin{align}
	\nonumber	\mathbb{E}\bigg((X-\mu_{0})g(X)\bigg)&=\mathbb{E}\int_{\mathbb R}ug(X+u)\nu_{vg}(du)\\
	&=\alpha\mathbb{E}\displaystyle\int_{0}^{\infty}\left( e^{-\lambda^{+}u}g(X+u) - e^{-\lambda^{-}u}g(X-u) \right)du.\label{PP3:c1}
	\end{align}
	\noindent
	Applying integration by parts formula two times on the right hand side of (\ref{PP3:c1}) and rearranging the integrals, we have
	\begin{align}
	\nonumber	\mathbb{E}(X-\mu_{0})g(X)	=&\alpha\left(\frac{1}{\lambda^{+}}-\frac{1}{\lambda^{-}}\right)\mathbb{E}g(X)+\frac{2\alpha}{\lambda^{+}\lambda^{-}}\mathbb{E}g^{\prime}(X)\\
	\nonumber	&+\alpha\left(\frac{1}{\lambda^{+}}-\frac{1}{\lambda^{-}}\right)\mathbb{E}\displaystyle\int_{0}^{\infty}\left(e^{-\lambda^{+}u}g^{\prime}(X+u)-e^{-\lambda^{-}u}g^{\prime}(X-u) \right)du\\
	&\quad+\frac{\alpha}{\lambda^{+}\lambda^{-}}\mathbb{E}\displaystyle\int_{0}^{\infty}\left(e^{-\lambda^{+}u}g^{\prime\prime}(X+u) -e^{-\lambda^{-}u}g^{\prime\prime}(X-u) \right)du.\label{PP3:c4}
	\end{align}
	\noindent
	Next taking $g^{\prime}$ and $g^{\prime \prime}$ in \eqref{PP3:c1}
	\begin{subequations}
		\begin{align}\label{PP3:c5}
		(a)~~\mathbb{E}(X-\mu_{0})g^{\prime}(X)&=\alpha\mathbb{E}\displaystyle\int_{0}^{\infty}\left(e^{-\lambda^{+}u}g^{\prime}(X+u) -e^{-\lambda^{-}u}g^{\prime}(X-u) \right)du,\\
		\label{PP3:c6} (b)~~\mathbb{E}(X-\mu_{0})g^{\prime\prime}(X)&=\alpha\mathbb{E}\displaystyle\int_{0}^{\infty}\left(e^{-\lambda^{+}u}g^{\prime\prime}(X+u) -e^{-\lambda^{-}u}g^{\prime\prime}(X-u) \right)du, 
		\end{align}
	\end{subequations}
	\noindent
	as $g\in\mathcal{S}(\mathbb{R})$. Now, using \eqref{PP3:c5} and \eqref{PP3:c6} on \eqref{PP3:c4}, we get	
	\begin{align}
	\nonumber	\mathbb{E}(X-\mu_{0})g(X)=&\mathbb{E}\left(\alpha\left(\frac{1}{\lambda^{+}}-\frac{1}{\lambda^{-}}\right)g(X) \right)+\frac{2\alpha}{\lambda^{+}\lambda^{-}}\mathbb{E} \left(g^{\prime}(X)\right)\\
	&+\left(\frac{1}{\lambda^{+}}-\frac{1}{\lambda^{-}}\right)\mathbb{E} \left( \left(X-\mu_{0}\right)g^{\prime}(X)\right)+\mathbb{E}\left( \frac{1}{\lambda^{+}\lambda^{-}}(X-\mu_{0})g^{\prime\prime}(X)\right).\label{PP3:c7}
	\end{align}
	\noindent
	Setting $\frac{1}{\lambda^{+}\lambda^{-}}=\sigma^{2}, \left(\frac{1}{\lambda^{+}}-\frac{1}{\lambda^{-}}\right)=2\theta,$ and $\alpha=\frac{r}{2}$, we get $\mathbb{E}(X)=\mu_0+r\theta$ and $Var(X)=r(\sigma^{2}+2\theta^{2})$. Also, for these parameters, (\ref{PP3:c7}) reduces to

	\vfill
	\begin{align}
	\label{e32}	&\mathbb{E}\bigg( \sigma^{2}(X-\mu_0)g^{\prime\prime}(X)+ \left(\sigma^{2}r+2\theta(X-\mu_0) \right)g^{\prime}(X)
	+\left(r\theta-(X-\mu_0)\right)g(X)\bigg)=0,
	\end{align}
	\noindent
	which is a Stein identity for $ VGD_{1}(\mu_0,\sigma^{2},r,\theta)$.
	
	
	\vskip 1ex
	\noindent
	Observe that the Stein identity given by Gaunt \cite{k24}, and our Stein identity given in \eqref{e32} match except for the choice of function space. Note that the derivation of Stein identity given in \cite{k24} is by modifying the density approach developed by Stein \textit{et al.} \cite{den}, and the density of $VGD$ is usually written in terms of modified Bessel functions. 	
\end{exmp}

\noindent
Finally, we obtain a Stein identity for $BGD$. This identity is, in a sense a generalization of the $VGD$ Stein identity, since $VGD_0(0,\alpha,\lambda^{+},\lambda^{-}) \overset{d}{=}BGD(\alpha,\lambda^{+},\alpha,\lambda^{-}).$	
\begin{exmp}(Bilateral-gamma distributions)\label{corBG}
	\noindent
	Let $X\sim BGD(\alpha^{+},\lambda^{+},\alpha^{-},\lambda^{-})$ whose L\'evy measure $\nu_{bg}$ is defined in \eqref{bgdlevy}. Then $BGD(\alpha^{+},\lambda^{+},\alpha^{-},\lambda^{-}) \overset{d}{=}IDD(\mu,0,$ $\nu_{bg})$, where $\mu=\mathbb{E}(X)=\int_{\mathbb{R}}u\nu_{bg}(du)$ and the rv $Y_1$ has density 
	$$f_1(v)=\frac{[\int_{v}^{\infty}u\nu_{bg}(du)]\mathbf{I}_{(0,\infty)}(v) + [\int_{-\infty}^{v}(-u)\nu_{bg}(du)]\mathbf{I}_{(-\infty,0)}(v)}{Var(X)},~v\in \mathbb{R}.$$
	
	\noindent
	Let $g \in \mathcal{S}(\mathbb{R})$. Then by Corollary \ref{theorem1},  
	$$Cov(X,g(X))=Var(X)\mathbb{E}\left(g^{\prime} (X+Y_1) \right)=\mathbb{E}\bigg(  \int_{\mathbb{R}}u(g(X+u)-g(X))\nu_{bg}(du)    \bigg),$$
	which can be written as
	\begin{align}
	\nonumber	\mathbb{E}\bigg(Xg(X)\bigg)&=\mathbb{E}\int_{\mathbb R}ug(X+u)\nu_{bg}(du)\\
	&=\mathbb{E}\displaystyle\int_{0}^{\infty}\left( \alpha^{+}e^{-\lambda^{+}u}g(X+u) -\alpha^{-} e^{-\lambda^{-}u}g(X-u) \right)du.\label{PP3:c10}
	\end{align}
	\noindent
	Following steps similar to the derivation of Stein identity for $VGD_1(\mu_{0},$
	$\sigma^{2},r,\theta)$ in Example \ref{corvg}, we get a Stein identity for $BGD(\alpha^{+},\lambda^{+},\alpha^{-},\lambda^{-})$ as

	\begin{align}
	\nonumber	\mathbb{E}&\bigg( Xg^{\prime\prime}(X)+ \left(\left(\alpha^{+}+\alpha^{-}\right)-\left(\lambda^{+}-\lambda^{-}\right)X \right)g^{\prime}(X)\\
	&\quad\quad+\left(\left(\alpha^{+}\lambda^{-}-\alpha^{-}\lambda^{+}\right)-\lambda^{+}\lambda^{-} X\right)g(X)\bigg)=0, ~~ g\in \mathcal{S}(\mathbb{R}).\label{BG1}
	\end{align}
	
	\noindent
	The identity in \eqref{BG1} has been recently and independently derived by Forrester \cite[Proposition 1]{gammadiff}, using the density approach, and also for a different function space.
	
\end{exmp}

\subsection{Application to weighted premium calculation principles}
In this section, we apply Theorem \ref{qtheorem1} and Corollary \ref{theorem1} to premium calculation principles. First, we obtain a new formula of the $WPCP$ given in \eqref{wpre0} for $IDD$, in terms of L\'evy measure via Stein-type covariance identity given in Corollary \ref{theorem1} which is novel in our opinion. Recall first that the $WPCP$ due to Furman and Zitikis \cite{furman} is as follows.

\begin{defn}\label{def1}
	Let $X$ be a loss rv of risk and $w:[0,\infty) \to [0,\infty) $ be a function such that $0<\mathbb{E}(w(X)) <\infty$. Then the $WPCP$ is defined as 
	\begin{align}\label{wpre}
	\mathcal{H}_{w}(X)=\frac{\mathbb{E}\left(Xw(X) \right) }{\mathbb{E}\left(w(X) \right)},
	\end{align}
\end{defn}
\noindent
which can also be rewritten as
\begin{align}\label{wpre0}
\mathcal{H}_w(X)=\mathbb{E}(X)+\frac{\text{Cov}(X,w(X))}{\mathbb{E}(w(X))}.
\end{align}
\begin{pro}
	Let $X\sim \text{IDD}(\mu,0,\nu)$ with $\mu=\mathbb{E}(X)=\int_{\mathbb R}u\nu(du) <\infty$ and $w:[0,\infty) \to [0,\infty) $ be a function such that $0<\mathbb{E}(w(X)) <\infty$. Then
	\begin{align}\label{wpre04}
		\mathcal{H}_w(X)=  \frac{\mathbb{E}\big(\int_{\mathbb{R}} w(X+u)u\nu(du)\big)}{\mathbb{E}(w(X)}.
	\end{align}	
	\end{pro}

	\begin{proof}
	 From Corollary \ref{theorem1}, we get 
		$$	\text{Cov}(X,g(X))=Var(X)\mathbb{E} \left(g^{\prime}(X+Y_1)\right),$$
		\noindent
		where the random variables $X$ and $Y_1$ are independent, with $pdf$ $f_{1}(y)$ of $Y_1$ given in \eqref{pdfzbias}. Now
		\begin{align}
		\nonumber\text{Cov}(X,g(X))=& Var(X)\mathbb{E} \left(g^{\prime}(X+Y_1)\right)\\
		\nonumber=&Var(X) \mathbb{E} \left(g^{\prime}(X+Y_1)\right)\\
		\nonumber=&Var(X)\mathbb{E}\displaystyle\int_{\mathbb{R}}g^{\prime}(X+v)f_{1}(v)dv\\
		\nonumber	=&\mathbb{E}\displaystyle\int_{\mathbb{R}}g^{\prime}(X+v) \left(\eta_{1}^{+}(v)\mathbf{I}_{(0,\infty)}(v)+ \eta_{1}^{-}(v)\mathbf{I}_{(-\infty,0)}(v)  \right)dv\\
		\nonumber=&\mathbb{E}\displaystyle\int_{0}^{\infty}g^{\prime}(X+v)\int_{v}^{\infty}u\nu(du) dv\\
		\nonumber&+\mathbb{E}\displaystyle\int_{-\infty}^{0}g^{\prime}(X+v)\int_{-\infty}^{v}(-u)\nu(du) dv\\
		\label{wpre1}=&\mathbb{E}\int_{\mathbb{R}} (g(X+u)-g(X))u\nu(du),
		\end{align}
		\noindent
		where the last equality is followed by Fubini's theorem and adjusting the integrals. Recall that, for any $w:[0,\infty) \to [0,\infty) $ with $0<\mathbb{E}(w(X)) <\infty$, the $WPCP$ (see \cite{furman}) is given by		
		\begin{align}
			\mathcal{H}_{w}(X)=\frac{\mathbb{E} \left( Xw(X)  \right) }{\mathbb{E}(w(X))}=\mathbb{E}(X)+\frac{Cov(X,w(X))}{\mathbb{E} (w(X))} .\label{wpre2}
		\end{align}
		
		\noindent
		 Next, replacing $g$ by $w$ in \eqref{wpre1} and substituting it in \eqref{wpre2}, we have
		\begin{align}
		\nonumber\mathcal{H}_w(X)&=\mathbb{E}(X)+\frac{\mathbb{E} \left(  \int_{\mathbb{R}} \left( w(X+u) -w(X)\right) u\nu (du)\right)}{\mathbb{E} (w(X))}\\
		  \label{wpre4}&=\frac{\mathbb{E}\left(\int_{\mathbb{R}} w(X+u)u\nu(du) \right)}{\mathbb{E}(w(X))},
		\end{align}
		\noindent
		which proves the result.	
	\end{proof}

	\begin{rem}
	Recently, Psarraakos \cite{geoP} proposed an alternate formula of $WPCP$ for a non-negative continuous rv using Stein-type covariance identity and its formula involves the $pdf$ of $X$. Hence the derivation of $WPCP$ is not straightforward, whenever the $pdf$ is not in closed form (or $pdf$ is, in terms of, some special functions, for example, bilateral-gamma \cite{bilateral} and variance-gamma distributions \cite{k24}). 
	\end{rem}

	\begin{exmp}(Compound Poisson distributions).
		Let $X\sim CPD(\nu(\mathbb{R}),\nu_0)$. That is, $X\sim IDD(\mu,0,\nu)$, where $\mu=\int_{\mathbb{R}}u\nu(du)$ and $\nu(du)=\nu(\mathbb{R})\nu_{0}(du)$. Then by \eqref{wpre04}, the $WPCP$ is given by

		$$\mathcal{H}_{w}(X)=\frac{\nu(\mathbb{R})\mathbb{E}\left( \int_{\mathbb{R}} w(X+u)u \nu_{0}(du) \right) }{\mathbb{E}(w(X))}.$$
		\noindent
		In particular, for $w(x)=e^{\kappa x},$ $\kappa>0$, $\nu(\mathbb{R})=\lambda$ and $\nu_{0}(du)=\delta_{1}(du)$, the $WPCP$ reduces to
		\begin{align*}
			\mathcal{H}_{w}(X)=\frac{\lambda \mathbb{E} \left( e^{\kappa (X+u)} u\delta_{1}(du)  \right) }{\mathbb{E} (e^{\kappa X})}=\lambda \int_{\mathbb{R}}e^{\kappa u} u\delta_{1}(du)=\lambda e^{\kappa},
		\end{align*}
		\noindent
		which is the Esscher principle for $Poi(\lambda)$ distribution, see \cite[p. 51]{dickson}. 
		
		\vskip 2ex
		\noindent
		Note that the covariance identity obtained in \cite{geoP} does not hold for discrete distributions. Therefore the $WPCP$ for Poisson distribution can not be retrieved from the covariance identity given in \cite{geoP}. 	
	\end{exmp}
	\noindent
	The following examples are new to the literature and also establish the importance of our approach using L\'evy measure.
	\begin{exmp}[$CGMY$ distributions] 
		Let $X\sim CGMY(\alpha,\beta,\lambda^{+},\lambda^{-})$ with non-zero mean. That is, $X\sim IDD(\mu,0,\nu_{cgmy})$, where $\mu=\mathbb{E}(X)=\int_{\mathbb{R}}u\nu_{cgmy}(du) $ and the L\'evy measure $\nu_{cgmy}$ is defined in \eqref{CGMYlevy}. Also, $Var(X)=\int_{\mathbb{R}}u^{2}\nu_{cgmy}(du)$. Then by \eqref{wpre04}, the $WPCP$ is given by
		\begin{align*}
			\mathcal{H}_{w}(X)&=\frac{\mathbb{E} \left(\int_{\mathbb{R}} w(X+u)u\nu_{cgmy}(du) \right) }{\mathbb{E}(w(X))}.
		\end{align*}
		\noindent
		\noindent
		Let $w(x)=x$ (known as the modified variance principle, see \cite[Section 2]{furman}), then
		\begin{align*}
			\mathcal{H}_{w}(X)=&\frac{\mathbb{E} \left(\displaystyle\int_{\mathbb{R}} (X+u)u\nu_{cgmy}(du) \right) }{\mathbb{E}(X)}\\
			=&\frac{\mathbb{E}(X)\displaystyle\int_{\mathbb{R}} u\nu_{cgmy}(du)+\displaystyle\int_{\mathbb{R}} u^{2}\nu_{cgmy}(du)}{\mathbb{E}(X)}	\\
			=&\mathbb{E}(X)+\frac{Var(X)}{\mathbb{E}(X)} \\
			=&\alpha \left(\frac{\Gamma(1-\beta)}{(\lambda^{+})^{1-\beta}}-  \frac{\Gamma(1-\beta)}{(\lambda^{-})^{1-\beta}} \right)+\frac{\frac{\Gamma(2-\beta)}{(\lambda^{+})^{2-\beta}} +\frac{\Gamma(2-\beta)}{(\lambda^{-})^{2-\beta}} }{\frac{\Gamma(1-\beta)}{(\lambda^{+})^{1-\beta}}-  \frac{\Gamma(1-\beta)}{(\lambda^{-})^{1-\beta}} }.
		\end{align*}
		
		\noindent
		 Also, for $w(x)=e^{\kappa x},~\kappa >0$, the Esscher principle, we have
		\begin{align*}
			\mathcal{H}_w (X)&=\frac{\mathbb{E} \left(\displaystyle\int_{\mathbb{R}} e^{\kappa (X+u)}u\nu_{cgmy}(du) \right) }{\mathbb{E}(e^{\kappa X})}\\
			&=\frac{\mathbb{E}(e^{\kappa X}) \left(\displaystyle\int_{\mathbb{R}}u e^{\kappa u}\nu_{cgmy}(du) \right) }{\mathbb{E}(e^{\kappa X})}\\
			&=\alpha \displaystyle\int_{0}^{\infty}ue^{\kappa u}\frac{e^{-\lambda^{+}u}}{u^{1+\beta}}du+\alpha \displaystyle\int_{-\infty}^{0}ue^{\kappa u}\frac{e^{-\lambda^{-}|u|}}{|u|^{1+\beta}}du\\
			&=\alpha \displaystyle\int_{0}^{\infty}\left(u^{-\beta}e^{-(\lambda^{+}-\kappa )u}- u^{-\beta}e^{-(\lambda^{-}+\kappa )u}\right)du   \\
			&=\frac{\alpha \Gamma (1-\beta)}{(\lambda^{+}- \kappa)^{1-\beta}}- \frac{\alpha \Gamma (1-\beta) }{(\lambda^{-}+\kappa)^{1-\beta}},~0<\kappa <\lambda^{+}.
		\end{align*}


	\end{exmp}	
\begin{exmp}[Bilateral-gamma distributions]
	 Let $X\sim BGD(\alpha^{+},\lambda^{+},\alpha^{-},\lambda^{-})$ with non-zero mean. That is, $X\sim IDD(\mu,0,\nu_{bg})$, where $\mu=\mathbb{E}(X)=\int_{\mathbb{R}}u\nu_{bg}(du)$ and the L\'evy measure $\nu_{bg}$ is defined in \eqref{bgdlevy}. Also, $Var(X)=\int_{\mathbb{R}}u^{2}\nu_{bg}(du)$. Then by \eqref{wpre04}, the $WPCP$ is given by
	\begin{align*}
		\mathcal{H}_{w}(X)&=\frac{\mathbb{E}\left( \int_{\mathbb{R}}w(X+u)u\nu_{bg}(du)\right)du  }{\mathbb{E}(w(X))}.
	\end{align*}
	\noindent
	For $w(x)=x$, the modified variance principle, we have
	\begin{align}
\nonumber	\mathcal{H}_{w}(X)&=\frac{\mathbb{E} \left(\displaystyle\int_{\mathbb{R}} (X+u)u\nu_{bg}(du) \right) }{\mathbb{E}(X)}\\
\nonumber	&=\displaystyle\int_{\mathbb{R}} u\nu_{bg}(du)+	\frac{\displaystyle\int_{\mathbb{R}} u^{2}\nu_{bg}(du)}{\displaystyle\int_{\mathbb{R}} u\nu_{bg}(du)}\\
	&= \left(\frac{\alpha^{+}}{\lambda^{+}}-  \frac{\alpha^{-}}{\lambda^{-}} \right)+\frac{\left(\frac{\alpha^+ }{(\lambda^{+})^{2}}+  \frac{\alpha^{-} }{(\lambda^{-})^{2}} \right)}{\left(\frac{\alpha^{+}}{\lambda^{+}}-  \frac{\alpha^{-}}{\lambda^{-}} \right)}.\label{wpcpbg0}
	\end{align}
	\noindent
	When $\alpha^+=\alpha^-=\alpha$, from \eqref{wpcpbg0}, we get
	\begin{align*}
		\mathcal{H}_{w}(X)&=\alpha\left(\frac{1}{\lambda^{+}}-  \frac{1}{\lambda^{-}} \right)+ \frac{\frac{1}{(\lambda^{+})^{2}}+  \frac{1 }{(\lambda^{-})^{2}} }{\frac{1}{\lambda^{+}}-  \frac{1}{\lambda^{-}} }\\
		&=\left(\lambda^+ \lambda^-  \right)^{-1} \left( \lambda^- -\lambda^+ \right)^{-1} \left( \alpha (\lambda^{-} - \lambda^{+})^{2}+ (\lambda^+)^2+(\lambda^-)^2  \right),
	\end{align*}
\noindent
which is the modified variance principle of $VGD_{0}(0,\alpha,\lambda^{+},\lambda^{-})$ distribution.

	\noindent
	Also, for $w(x)=e^{\kappa x},~\kappa >0$, the Esscher principle, we have
	\begin{align}
\nonumber	\mathcal{H}_w (X)&=\frac{\mathbb{E} \left(\displaystyle\int_{\mathbb{R}} e^{\kappa (X+u)}u\nu_{bg}(du) \right) }{\mathbb{E}(e^{\kappa X})}\\
\nonumber	&=\frac{\mathbb{E}(e^{\kappa X}) \left(\displaystyle\int_{\mathbb{R}}u e^{\kappa u}\nu_{bg}(du) \right) }{\mathbb{E}(e^{\kappa X})}\\
\nonumber	&=\alpha^{+} \displaystyle\int_{0}^{\infty}e^{-(\lambda^{+}-\kappa )u}du- \alpha^{-}\displaystyle\int_{0}^{\infty}e^{-(\lambda^{-}+\kappa )u}du   \\
	&=\frac{\alpha^{+}}{\lambda^{+}- \kappa}- \frac{\alpha^{-}}{\lambda^{-}+\kappa},~0<\kappa <\lambda^{+}.\label{wpcpbg1}
	\end{align}	
	\noindent
	Also, when $\alpha^+=\alpha^-=\alpha,$ from \eqref{wpcpbg1}, we get
		$$\mathcal{H}_w(X)=\alpha\left(\frac{1}{\lambda^{+}- \kappa}- \frac{1}{\lambda^{-}+\kappa}\right),~0<\kappa <\lambda^{+},$$
\noindent		
which is the Esscher principle of $VGD_{0}(0,\alpha,\lambda^{+},\lambda^{-})$ distribution.
	
\end{exmp}	
\begin{exmp}[Inverse Gaussian distribution] Let $X$ have an inverse Gaussian distribution with parameters $\alpha,\lambda>0$ with $pdf$

	$$f_{X}(x)=\frac{\alpha}{x^{\frac{3}{2}}}e^{2\alpha \sqrt{\pi \lambda} -\lambda x - \pi \alpha^{2}/ x},~x>0.$$
\noindent
That is, $X\sim IDD(\mu,0,\nu_{ig})$, where $\mu=\int_{\mathbb{R}}u\nu_{ig}(du)$ and the L\'evy measure $\nu_{ig}$ is given by $\nu_{ig}(du)=\alpha e^{-\lambda u}u^{-\frac{3}{2}}\mathbf{I}_{(0,\infty)}(u)du$ (see \cite[Section 4.4.2]{cont}). Then by \eqref{wpre04}, the $WPCP$ is given by

$$\mathcal{H}_{w}(X)=\frac{\alpha\mathbb{E} \int_{0}^{\infty}e^{-\lambda u}u^{-\frac{1}{2}} w(X+u)du}{\mathbb{E}(w(X))}.$$

\noindent
For $w(x)=x$, the modified variance principle, we have
\begin{align*}
	\mathcal{H}_{w}(X)=\frac{\mathbb{E} \left(\displaystyle\int_{\mathbb{R}} (X+u)u\nu_{ig}(du) \right) }{\mathbb{E}(X)}=\displaystyle\int_{\mathbb{R}} u\nu_{ig}(du)+	\frac{\displaystyle\int_{\mathbb{R}} u^{2}\nu_{ig}(du)}{\displaystyle\int_{\mathbb{R}} u\nu_{ig}(du)}=\frac{1+2\alpha \sqrt{\pi}}{2 \lambda}.
\end{align*}
\noindent
 Also, for $w(x)=e^{\kappa x},~\kappa >0$, the Esscher principle, we have
 $$\mathcal{H}_{w}(X)=\frac{\alpha \sqrt{\pi}}{(\lambda -\kappa)^{\frac{3}{2}}},~0<\kappa<\lambda.$$
 	
\end{exmp}

	\subsubsection{Generalized $WPCP$} Here, we apply Theorem \ref{qtheorem1} to generalized $WPCP$. We first recall the following definition; see \cite[Section 4]{furman} for more details.
	\begin{defn}
		Let $X$ be a rv of risk and $w:[0,\infty)\to [0,\infty)$ be an increasing function. Then for an increasing non-negative function $g$, the generalized $WPCP$ is defined as
		$$\mathcal{H}_{g,w}(X)=\frac{\mathbb{E}(g(X)w(X))}{\mathbb{E}(w(X))}.$$ 
	\end{defn}
	
	\noindent
	In particular, using \eqref{qeqn1} with $g(x)=x^{n},~n\geq1$, we have the following result.
	
	\begin{pro}
		Let $X\sim IDD(\mu,0,\nu)$ with $\mu=\int_{\mathbb{R}}u\nu(du)<\infty$ and $w:[0,\infty)\to [0,\infty)$ be a function such that $0<\mathbb{E}(w(X))<\infty$. Suppose that for $x\in[0,\infty)$, $g(x)=x^{n}$, $n\geq1$ . Then 
		\begin{align}\label{gwpcp}
		\mathcal{H}_{x^{n},w}(X)=\mathbb{E}(X^{n})+\frac{\displaystyle\sum_{k=0}^{n-1}\binom{n}{k}\displaystyle\int_{0}^{1}\mathbb{E}\left(Y_{s}^{k}\displaystyle\int_{\mathbb{R}}w^{\prime}(X_s+v)\eta_{n-k}(v)dv\right) ds}{\mathbb{E}(w(X))},
		\end{align}	
		where the random vector $(X_s,Y_s)$ have the cf given in \eqref{jointcf} and $\eta_{k}$'s are defined in \eqref{qnnot0}.	
	\end{pro}
	\noindent
	Next, we discuss an example of generalized $WPCP$ using the above formula.
	
	\begin{exmp}(Gamma distribution)
		Let $X\sim Ga(a,b)$ ($a,b>0$) with $pdf$
		\begin{align}
		f_{X}(x)=\frac{b^{a}}{\Gamma(a)}x^{a-1}e^{-bx},~x>0.
		\end{align}
		\noindent
		Let $\Gamma(s,x)=\int_{s}^{\infty}e^{-t}t^{x-1}dt$ ($s,x>0$) be the incomplete gamma function. Let $w:[0,\infty)\to [0,\infty)$ be a function such that $0<\mathbb{E}(w(X))<\infty$. Suppose also that for $x\in[0,\infty)$, $g(x)=x^{n}$, $n\geq1$. Then by \eqref{gwpcp}, the generalized $WPCP$ is given by
		\begin{align}\label{gwpcp00}
		\mathcal{H}_{x^{n},w}(X)=\frac{\Gamma (a+n)}{\Gamma(a)b^{n}}+ \frac{\displaystyle\sum_{k=0}^{n-1}\binom{n}{k}\displaystyle\int_{0}^{1}\mathbb{E}\left(Y_{s}^{k}\displaystyle\int_{0}^{\infty}w^{\prime}(X_s+v)\eta_{n-k}(v)dv\right) ds}{\mathbb{E}(w(X))},
		\end{align}
		
		\noindent
		where $\eta_{k}$'s are defined by, for any $y>0$,
		$$\eta_{k}(y)= \frac{a}{b^k}\Gamma (by,k),~k\geq 1.$$
	\noindent
	Note also that $(X_s,Y_s)$ has the cf
	\begin{align}
		\phi_{s}(t,z)=\bigg(1-\frac{it}{b}\bigg)^{-(1-s)a}\bigg(1-\frac{iz}{b}\bigg)^{-(1-s)a}\bigg(1-\frac{i(t+z)}{b}\bigg)^{-sa},~t,z\in \mathbb{R},
	\end{align}
\noindent
since $\phi_{X}(t)=\bigg(1-\frac{it}{b}\bigg)^{-a}$.	
		
\noindent
Also, by Corollary \ref{pcid}, the formula in \eqref{gwpcp00} can be seen as
\begin{align}\label{gwpcp0}	
	\mathcal{H}_{x^{n},w}(X)=\frac{\Gamma (a+n)}{\Gamma(a)b^{n}}+ \frac{(n!)}{b^{n+1}} \frac{\displaystyle\sum_{k=0}^{n-1}\frac{ab^k}{k!}\displaystyle\int_{0}^{1} \mathbb{E}\left(Y_s^{k}w^{\prime}(X_s+Y_{n-k})\right)ds }{\mathbb{E}(w(X))},	
	\end{align}
	\noindent
	where the rv $Y_k$ has the density, for any $y>0$,
	$$f_k(y)=\frac{b}{\Gamma(k+1)}\Gamma(by,k),~k\geq 1.$$
\noindent
For example, when $n=1$ and $w(x)=e^{\kappa x},~\kappa >0$, from \eqref{gwpcp0}, we get
\begin{align}\label{gwp000}
	\mathcal{H}_{x,w}(X)&=\frac{a}{b}+\frac{\kappa a}{b^2}\frac{\int_{0}^{1}\mathbb{E}\left(e^{\kappa(X_s+Y_{1})}\right)ds }{\mathbb{E} (e^X)},
\end{align}	
\noindent
where $Y_1$ has density
$$f_1(y)=be^{-by},~y>0.$$ Since $X_s \overset{d}{=}X,$ then \eqref{gwp000} reduces to

	\begin{align}
\nonumber	\mathcal{H}_{x,w}(X)&=\frac{a}{b}+\frac{\kappa a}{b^2}\frac{\mathbb{E}\left(e^{\kappa (X+Y_{1})}\right) }{\mathbb{E} (e^{\kappa X})}\\
\nonumber&=\frac{a}{b}+\frac{\kappa a}{b^2}\mathbb{E} (e^{\kappa Y_1})\\
\nonumber	&=\frac{a}{b}+\frac{\kappa a}{b}\int_{0}^{\infty}e^{-(b-\kappa)y}dy\\
\nonumber	&=\frac{a}{b}+\frac{\kappa a}{b (b-\kappa)}\\
	\label{wpcpgama}&=\frac{a}{b- \kappa},~0<\kappa <b,
	\end{align}
	\noindent
	which is the Esscher principle of gamma $Ga(a,b)$ distribution.
	
	\vskip 1ex
	\noindent
	Observe that, for $a=b=1$, the $WPCP$ in \eqref{wpcpgama} matches exactly with the Esscher principle of $Ga (1,1)$ distribution (exponential distribution with mean $1$)  (see \cite[Example 3.1]{dickson}), where the author uses density function to derive the premium calculation principle. 	
\end{exmp}

	\subsection{Application to Gini coefficient}
	Here, we obtain an alternate formula of the Gini coefficient for $IDD$, in terms of L\'evy measure via Stein-type covariance identity. Recall that the Gini coefficient is defined as follows (see \cite{Tz}). 
	\begin{defn}
		Let $X$ have the distribution $F_X(x)$. The Gini coefficient of the distribution function $F_X(x)$ is defined as
		\begin{align}\label{Gini1}
		G=\frac{2}{\mathbb{E}(X)}\text{Cov}(X,F_X(X)).
		\end{align}
	\end{defn}

	\begin{pro}
	Let $X\sim \text{IDD}(\mu,0,\nu)$ and $F_{X}(x)$ be the distribution of $X$. Then for $\mu>0$,
	\begin{align}\label{Gini4}
	G=\frac{2}{\mu}\mathbb{E}\int_{\mathbb{R}} u(F_X(X+u)-F_X(X))\nu(du).
	\end{align}	
	\end{pro}

	\begin{proof}
	\noindent
	From Corollary \ref{theorem1}, it can be shown that	
	\begin{align}\label{ginicov1}
	\text{Cov}(X,g(X))&=\mathbb{E}\int_{\mathbb{R}} u(g(X+u)-g(X))\nu(du).
	\end{align}
	
	\noindent
	Replacing $g$ by $F_X$ in \eqref{ginicov1} and substituting in \eqref{Gini1}, the desired conclusion follows.
	\end{proof}
	
	\noindent
	Next, we discuss some examples.
	
	\begin{exmp}[$CGMY$ distribution] Let $X\sim CGMY(\alpha, \beta,\lambda^{+},\lambda^{-})$. Then the mean is given by
		$$\mathbb{E}(X)=\mu=\Gamma(1-\beta) \frac{\alpha}{(\lambda^{+})^{1-\beta}}-\Gamma(1-\beta) \frac{\alpha}{(\lambda^{-})^{1-\beta}}.$$
		
		\noindent
		Then using \eqref{Gini4}, the Gini coefficient is given by
		
		$$G_{\text{CGMY}}=\frac{2}{\mu}\left(\Gamma(2-\beta) \frac{\alpha}{(\lambda^{+})^{2-\beta}}+\Gamma(2-\beta) \frac{\alpha}{(\lambda^{-})^{2-\beta}}\right).$$
	\end{exmp}
	
	\begin{exmp}[Bilateral gamma distribution] Let $X\sim BGD(\alpha^{+},\lambda^{+},\alpha^{-},\lambda^{-})$. Then the mean is given by
		
		$$\mathbb{E}(X)=\mu=\frac{\alpha^{+}}{\lambda^{+}}-\frac{\alpha^{+}}{\lambda^{+}}.$$
		
		\noindent
		Then using \eqref{Gini4}, the Gini coefficient is given by
		
		$$G_{\text{BGD}}=\frac{2}{\mu}\left( \frac{\alpha^{+}}{(\lambda^{+})^{2}}+\frac{\alpha^{-}}{(\lambda^{-})^{2}}\right).$$
	\end{exmp}

	\begin{exmp}[Variance gamma distribution] Let $X\sim VGD_0(\mu_{0},\alpha,\lambda^{+},\lambda^{-})$. Then the mean is given by
		
		$$\mathbb{E}(X)=\mu=\mu_{0}+\frac{\alpha}{\lambda^{+}}-\frac{\alpha}{\lambda^{+}}.$$
		
		\noindent
		Then using \eqref{Gini4}, the Gini coefficient is given by
		$$G_{\text{VGD}}=\frac{2 \alpha}{\mu}\left( \frac{1}{(\lambda^{+})^{2}}+\frac{1}{(\lambda^{-})^{2}}\right).$$
	\end{exmp}

	\noindent 
	\textbf{Acknowledgment}: The authors are thankful to the reviewer for the several helpful comments and suggestions, and also pointing out an error in our earlier version.

	\setstretch{.8}


\end{document}